\documentclass[a4paper]{article}

\usepackage[pages=all, color=black, position={current page.south}, placement=bottom, scale=1, opacity=1, vshift=5mm]{background}

\usepackage[margin=1in]{geometry} 

\usepackage{amsmath}
\usepackage{amsthm}
\usepackage{amssymb}

\usepackage[utf8]{inputenc}
\usepackage{hyperref}
\usepackage{todonotes}
\usepackage{graphicx}
\usepackage{verbatim,xcolor}
\usepackage{bm, tikz, color, relsize}
\usepackage{multirow}
\usepackage{algorithmicx}
\usepackage[ruled]{algorithm}
\usepackage{algpseudocode}
\usepackage[symbol]{footmisc}
\usetikzlibrary{arrows}

\usepackage[sort&compress,numbers,square]{natbib}
\theoremstyle{plain}
\newtheorem{theorem}{Theorem}[section]
\newtheorem{corollary}[theorem]{Corollary}
\newtheorem{lemma}[theorem]{Lemma}

\newtheorem{proposition}[theorem]{Proposition}

\theoremstyle{definition}

\newtheorem{example}[theorem]{Example}
\newtheorem{remark}[theorem]{Remark}

\usepackage{graphicx, color}
\graphicspath{{fig/}}
\usepackage{algorithm, algpseudocode}
\usepackage{mathrsfs}

\usepackage{lipsum}
\numberwithin{equation}{section}

\newcommand{\ba}{{\bf a}}

\newcommand{\bx}{{\bf x}}
\newcommand{\bn}{{\bf n}}

\newcommand{\bv}{{\bf v}}

\newcommand{\bw}{{\bf w}}

\newcommand{\Th}{\mathcal{T}_h}
\newcommand{\Kh}{\mathcal{K}_h}

\newcommand{\Eho}{\mathcal{E}_h^o}

\newcommand{\norm}[1]{\lVert #1\rVert}

\newcommand{\snorm}[1]{|#1|}
\newcommand{\trinorm}[1]{{\left\vert\kern-0.25ex\left\vert\kern-0.25ex\left\vert #1 \right\vert\kern-0.25ex\right\vert\kern-0.25ex\right\vert}}

\title{Edge-averaged virtual element methods for convection-diffusion and convection-dominated problems\footnote{Submitted to the editors in 2024.}}
\author{Shuhao Cao,\thanks{School of Science and Engineering, University of Missouri-Kansas City, Kansas City, MO 64110 (\texttt{scao@umkc.edu})}
\and Long Chen,\thanks{Department of Mathematics, University of California at Irvine, Irvine, CA 92697 (\texttt{lchen7@uci.edu})}
\and Seulip Lee,\thanks{Department of Mathematics, University of Georgia, Athens, GA 30602 (\texttt{seulip.lee@uga.edu})}}

\date{
}

\begin{document}
	\maketitle
	
	\begin{abstract}
		This manuscript develops edge-averaged virtual element (EAVE) methodologies to address convection diffusion problems effectively in the convection-dominated regime. It introduces a variant of EAVE that ensures monotonicity (producing an $M$-matrix) on Voronoi polygonal meshes, provided their duals are Delaunay triangulations with acute angles. Furthermore, the study outlines a comprehensive framework for EAVE methodologies, introducing another variant that integrates with the stiffness matrix derived from the lowest-order virtual element method for the Poisson equation. Numerical experiments confirm the theoretical advantages of the monotonicity property and demonstrate an optimal convergence rate across various mesh configurations.
	\end{abstract}

\section{Introduction}\label{sec:intro}

We consider the convection-diffusion equation in a bounded domain $\Omega\subset\mathbb{R}^2$ with the simply-connected polygonal Lipschitz boundary $\partial\Omega$:
\begin{subequations}\label{sys: governing}
\begin{alignat}{2}
-\nabla\cdot\left(\alpha(\mathbf{x}) \nabla u+\bm{\beta}(\mathbf{x})u\right) & = f && \quad \text{in } \Omega, \label{eqn: governing1} \\
u &= 0 && \quad \text{on } \partial\Omega,  \label{eqn: governing2}
\end{alignat}
\end{subequations}
where $\alpha\in C^0(\bar{\Omega})$ with $0<\alpha_{\text{min}}\leq\alpha(\mathbf{x})\leq \alpha_{\text{max}}$ for every $\mathbf{x}\in \Omega$, $\bm{\beta}\in (C^0(\bar{\Omega}))^2$, and $f\in L^2(\Omega)$.
The convection-diffusion problem~\eqref{sys: governing} expresses the convective and molecular transport along a stream moving at the given velocity $\bm{\beta}$ and with the diffusive effect from $\alpha$.
The convection-dominated regime means the situation that
\begin{equation*}
\alpha(\mathbf{x}) \ll\left|\bm{\beta}(\mathbf{x})\right|,\quad\forall \mathbf{x}\in\Omega,
\end{equation*}
and the convection-diffusion problem in this regime is important in describing viscous fluid models with high Reynolds numbers governed by the Navier-Stokes equations.

It has long been known that the standard-conforming finite element method (FEM) fails to provide accurate numerical solutions for convection-dominated problems.
The finite element solutions based on standard Galerkin projections contain spurious oscillations that deteriorate the solutions' quality (e.g., see \cite{1977Kikuchi,ELM05,STY05,2007JohnKnobloch}).
Hence, various stabilized FEMs have been developed in two primary directions.
The first direction is based on presenting a provable order of convergence in appropriate norms independent of small $\alpha$ by adding stabilization terms to the weak formulation.
This category includes the streamline-upwind Petrov-Galerkin (SUPG) methods~\cite{BRO82}, continuous interior penalty methods~\cite{BUR04}, and local projection stabilization (LPS) methods~\cite{KNO10}.
This stabilization technique mitigates such unexpected oscillations but does not eliminate them.
The second direction is the development of stabilized methods satisfying the discrete maximum principle (DMP).
When solving a second-order elliptic equation numerically, it is known that a numerical scheme satisfying the DMP guarantees discrete solutions without spurious oscillations (see, e.g., \cite{1973CiarletRaviart,2005BurmanErn,2008Roos}).
Among many examples, the algebraic flux corrections (AFC) method~\cite{BAR18} is an upwind-based discretization that directly satisfies the DMP with a reduced artificial diffusion effect, but it requires solving a nonlinear system.
The monotonicity property~\cite{gilbarg1977elliptic} is a sufficient condition for the DMP, so monotone schemes have received attention despite their first-order convergence when $\alpha=0$~\cite{godunov1959difference}.
Moreover, getting a non-singular $M$-matrix in a discretization implies the monotonicity property. 
Such an $M$-matrix is characterized by nonpositive off-diagonal entries and diagonal dominance, including at least one strictly diagonally dominant column.
We note that this $M$-matrix condition makes it easier to validate the DMP.

The edge-averaged finite element (EAFE) method~\cite{XU99,LAZ05} is a well-known linear monotone FEM, and it has been recently applied to space-time discretization~\cite{bank2017arbitrary} and nearly inviscid incompressible flows~\cite{li2022new}.
The EAFE method is monotone if the stiffness matrix of the linear conforming FEM for the Poisson equation is an $M$-matrix, achieved by a Delaunay triangulation.
The EAFE method has been motivated by prior studies of the Scharfetter-Gummel method~\cite{scharfetter1969large} (e.g., the finite volume Scharfetter-Gummel method~\cite{bank1998finite}, the exponential fitting method~\cite{brezzi1989numerical, brezzi1989two, dorfler1999uniform}, and the inverse-average-type finite element~\cite{markowich1988inverse}).
This idea has recently been extended to a general framework for scalar and vector convection-diffusion problems, called the simplex-averaged finite element (SAFE) methods~\cite{wu2020simplex}.
A stable mimetic finite difference (MFD) method~\cite{adler2022stable} has also been presented in the SAFE framework.
However, none of these methods work for general polygonal meshes because it may not be straightforward to translate the conformity of the polynomial elements on simplices to polygons.
For this reason, our work proposes a generalization of the EAFE method compatible with virtual element methods on polygonal meshes.

Virtual element methods (VEMs), a generalization of FEMs, are novel numerical PDE methods on general polygonal and polyhedral meshes.
The VEMs use a non-polynomial approximation in a polygonal element, and exact pointwise values of local basis functions are not needed inside the element. Instead the degrees of freedom (DoFs) can be used to compute all necessary quantities to build a stable and accurate discretization.
More details can be found in \cite{BEI13,BEI14} and Section 3 in this paper.

This paper specializes in developing stabilized virtual element methods for convection dominated problems.
The SUPG stabilization and local projection stabilization (LPS) have been successfully applied to the VEMs in \cite{BEI20,BEN16,BER18} and \cite{LI21}, respectively.
However, these approaches do not guarantee the DMP, which may yield undesired oscillations in discrete solutions.
Therefore, we focus on generalizing the EAFE stabilization to the VEMs and call the resulting stabilized methods \textit{edge-averaged virtual element} (EAVE) methods.
We present two variants of the EAVE methods: \textbf{(1)} a monotone EAVE method on Voronoi meshes with dual Delaunay and \textbf{(2)} a general EAVE framework working with general polygonal meshes.

\textbf{(1)} The monotone EAVE method guarantees the DMP on Voronoi meshes with dual Delaunay triangulations consisting of acute triangles.
The basic idea of deriving this method is utilizing flux approximations and dual edge patches for mass lumping inspired by \cite{XU99} and \cite{brezzi2006error}, respectively.

\textbf{(2)} The general EAVE framework presents another EAVE method on general polygonal meshes. Its monotonicity property holds if the stiffness matrix of the linear VEM for the Poisson equation is an $M$-matrix.
To derive a bilinear form in this framework, we define a flux approximation in the lowest-order $H(\textbf{curl})$-conforming virtual element space and observe the relationship between the lowest-order $H(\textbf{curl})$- and linear $H^1$-conforming spaces.

Through numerical experiments, we see that the monotone EAVE method produces a more stable and effective solution than the general framework on the same Voronoi meshes.
The numerical results also demonstrate that all the EAVE methods have first-order convergence on various polygonal meshes. Finally, we compare the numerical performance of the EAVE methods with other stabilized numerical methods and show the robustness and first order accuracy of EAVE.

The remaining sections are structured:
Section~\ref{sec:pre} introduces necessary notations and observes the EAFE stabilization's main idea.
The essential definitions and theoretical aspects of the VEMs are introduced in Section~\ref{sec:VEM}.
In Section~\ref{sec: fvm_vem}, we present a finite-volume bilinear form for the Poisson equation compatible with a virtual element space and stabilization-free, providing fundamental ideas of using dual edge patches for mass lumping.
Section~\ref{sec: MEAVE} proposes a monotone EAVE method on Voronoi meshes with dual Delaunay triangulations consisting of acute triangles.
Section~\ref{sec:EAVE} presents a general framework for the EAVE methods and demonstrates a sufficient condition for the monotonicity property.
In Section~\ref{sec: numerical}, our numerical experiments show the effect of edge-averaged stabilization on VEMs as well as the optimal order of convergence on different mesh types.
Finally, we summarize our contribution in this paper and discuss related research in Section~\ref{sec: conclusion}.


\section{Preliminaries} \label{sec:pre}

To begin with, we introduce some notation used throughout this paper.
For a bounded Lipschitz domain $\mathcal{D}\in\mathbb{R}^2$, we denote the Sobolev space as $W^{m,p}(\mathcal{D})$ for $m\geq 0$ and $p\geq 1$.
Its norm and seminorm are denoted by $\|\cdot\|_{W^{k,p}(\mathcal{D})}$ and $|\cdot|_{W^{k,p}(\mathcal{D})}$, respectively.
When $p=2$, the Sobolev space is written as $H^m(\mathcal{D})$, and the space $H^0(\mathcal{D})$ coincides with $L^2(\mathcal{D})$.
The notation $H_0^1(\mathcal{D})$ means the space of $v\in H^1(\mathcal{D})$ such that $v=0$ on $\partial\mathcal{D}$ in the trace sense.
The polynomial spaces of degrees less than or equal to $k$ are denoted as $\mathcal{P}_k(\mathcal{D})$, and the space of continuous functions is denoted as $C^0(\mathcal{D})$.


\subsection{Model problem}
In the convection-diffusion equation~\eqref{eqn: governing1}, we denote the flux as
\begin{equation}
J(u) := \alpha \nabla u+\bm{\beta}u.\label{flux}
\end{equation}
The weak formulation of the problem~\eqref{sys: governing} is to find $u\in H_0^1(\Omega)$ such that
\begin{equation}
B(u,v)=F(v),\quad \forall v\in H_0^1(\Omega),\label{weakform}
\end{equation}
where the bilinear form and the functional on the right-hand side are defined as
\begin{equation*}
B(w,v):=\int_\Omega J(w)\cdot \nabla v\;{\rm d}\mathbf{x}\quad\text{and}\quad F(v) :=\int_\Omega fv\;{\rm d}\mathbf{x}.
\end{equation*}
The well-posedness of the continuous problem~\eqref{weakform} can be found in \cite{gilbarg1977elliptic}.

\subsection{Exponential fitting of the flux}
If there exists a potential function $\psi$ such that $
\nabla\psi  = \alpha^{-1}\bm{\beta},$
then the flux $J(u)$ can be expressed as a diffusion flux with a variable coefficient,
\begin{equation*}
J(u) = \kappa\nabla (e^{\psi} u)\quad \text{with}\quad\nabla\psi  = \alpha^{-1}\bm{\beta}, \quad\kappa(\bx)=\alpha(\bx) e^{-\psi(\bx)}.
\end{equation*}
In terms of differential forms, $e^{\psi}u$ is a 0-form, $\nabla(e^{\psi}u)$ is a 1-form, $J(u)$ is a 2-form, and $\kappa$ is a Hodge star mapping a 1-form to a 2-form.

If $\alpha$ and $\bm{\beta}$ are constant, the function $\psi$ is defined as  $\psi (\mathbf{x})=\alpha^{-1}\bm{\beta}\cdot\mathbf{x}$ and it is unique up to a constant.
However, such a potential function $\psi$ may not exist for general $\alpha$ and $\bm{\beta}$ if $\text{curl}\;(\alpha^{-1}\bm{\beta})\not=0$.

A proper flux approximation~\cite{XU99, LAZ05} is defined in the lowest order N\'ed\'elec space on triangular meshes, and it has been successfully applied to develop the edge-averaged finite element (EAFE) scheme.
In what follows, we briefly explain the main idea of the EAFE scheme on a triangulation.

\subsection{Flux approximations}
Let $\mathcal{T}_h$ be a conforming triangulation for $\Omega$ and $V_h=\mathcal{P}_1(\mathcal{T}_h)$ be the piecewise linear $H^1$-conforming finite element space. 
We consider the local flux $J(u_h)|_T=(\alpha \nabla u_h+\bm{\beta}u_h)|_T$ for $u_h\in V_h$ and $T\in\Th$. We define $\bm{\tau}_E$ as a scaled tangent vector with $|\bm{\tau}_E|=|E|$.
Suppose $\bx(t)$ is a parametrization of $E$ such that $\bx(0)=\bx_i$ and $\bx(1)=\bx_j$ are two endpoints of $E$. Let us define
\begin{equation}\label{eq:edgepotential}
\psi_E(\bx(t)) = \int_0^t \frac{1}{|E|}\alpha^{-1}(\bm{\beta}\cdot \bm{\tau}_E)\; {\rm d} s, \quad \quad\kappa_E=\alpha e^{-\psi_E}>0.
\end{equation}
Then, it follows
$$
\left.J(u_h)\cdot\bm{\tau}_E\right|_E=\left.\kappa_E\nabla(e^{\psi_E}u_h)\cdot\bm{\tau}_E\right|_E.
$$ 
That is, along any line segment, the flux in that direction is a diffusion flux with variable coefficients $\kappa_E$.
Note that $\psi_E$ is unique up to a constant, and $\left.J(u_h)\cdot\bm{\tau}_E\right|_E$ is invariant to the constant change.

The local N\'ed\'elec space of the lowest order is defined as
\begin{equation*}
\mathcal{N}_0(T):=(\mathcal{P}_0(T))^2+\mathbf{x}^\perp\mathcal{P}_0(T),
\end{equation*}
where $\mathbf{x}^\perp$ is a $90^{\circ}$ rotation of $\mathbf{x}$ satisfying $\mathbf{x}\cdot\mathbf{x}^\perp=0$.
Its DoFs for any $\bv\in \mathcal{N}_0(T)$ are defined as
\begin{equation*}
\text{dof}_E(\mathbf{v})=\frac{1}{|E|}\int_E\mathbf{v}\cdot\bm{\tau}_E\;{\rm d}s=\left.(\mathbf{v}\cdot\bm{\tau}_E)\right|_E.
\end{equation*}
Its canonical basis functions $\{\bm{\chi}_{E_j}\}_{j=1}^3\subset \mathcal{N}_0(T)$ are obtained from $\text{dof}_{E_i}(\bm{\chi}_{E_j})=\delta_{ij}$ for $1\leq i,j\leq 3$.
If we apply $\text{dof}_E$ to $\kappa_E^{-1}J(u_h)|_T$, then we have
\begin{equation}
\text{dof}_{E}(\kappa_E^{-1}J(u_h))=\text{dof}_{E}(\nabla(e^{\psi_E} u_h))=\delta_{E}(e^{\psi_E} u_h):=(e^{\psi_E} u_h)(\bx_j)-(e^{\psi_E} u_h)(\bx_i),
\label{def delta}
\end{equation}
where $\mathbf{x}_i$ and $\mathbf{x}_j$ are the endpoints of $E$, that is, $\bm{\tau}_E=\bx_j-\bx_i$.
In the terminology of differential forms, the flux $J(u_h)$ is a $2$-form, and $\kappa_E^{-1}J(u_h)$ is a $1$-form. The functional $\text{dof}_E$ is applied to the $1$-form.

Define $\bar{\kappa}_E$ as the harmonic average of $\kappa_E=\alpha e^{-\psi_E}$ on $E$, i.e.,
\begin{equation}
\bar{\kappa}_E=\left(\frac{1}{|E|}\int_{E}\kappa_E^{-1}\;{\rm d}s\right)^{-1}.\label{harmonic average}
\end{equation}
We look for a flux approximation $J_T(u_h)\in \mathcal{N}_0(T)$ written as
\begin{equation*}
J_T (u_h)= \sum_{E\subset \partial T}\text{dof}_{E}(\kappa_E^{-1}J(u_h)) \bar{\kappa}_{E}\bm{\chi}_E=\sum_{E\subset \partial T}\delta_{E}(e^{\psi_E} u_h)\bar{\kappa}_{E}\bm{\chi}_E,
\end{equation*}
where $\bar{\kappa}_E$ can be considered an edgewise Hodge star mapping a 1-form $\bm{\chi}_E$ to a 2-form.
The flux approximation $J_T(u_h)$ is also invariant to the constant shift of $\psi_E$ to $\psi_E+c_E$.

Moreover, we can express $\nabla v_h$ in $\mathcal N_0(T)$ as 
$$
\nabla v_h=\sum_{E\subset \partial T} (\nabla v_h\cdot\bm{\tau}_E) \bm{\chi}_E = \sum_{E\subset \partial T}\delta_{E}(v_h)\bm{\chi}_E.
$$
As a result, the following auxiliary bilinear form is used in the discretization
\begin{align}
B^{T}_\mathcal{N}(u_h,v_h):=&\int_T J_T(u_h)\cdot\nabla v_h\;{\rm d}\mathbf{x}\nonumber\\
=&\sum_{E_i,E_j\subset \partial T}\bar{\kappa}_{E_i}\delta_{E_i}(e^{\psi_{E_i}} u_h)\delta_{E_j}(v_h) \int_T\bm{\chi}_{E_i}\cdot \bm{\chi}_{E_j}\;{\rm d}\mathbf{x}. \label{inter bilinear form}
\end{align}
The mass matrix $(\int_T\bm{\chi}_{E_i}\cdot \bm{\chi}_{E_j}\;{\rm d}\mathbf{x})$ for $\mathcal N_0(T)$ is in general a full matrix. We shall apply a mass lumping technique to make it diagonal.


\subsection{Edge-averaged finite element scheme}\label{EAFE}
To derive the EAFE scheme~\cite{XU99,LAZ05}, we approximate the bilinear form~\eqref{inter bilinear form} using mass lumping.
We let $\{\lambda_i(\mathbf{x})\}_{i=1}^3\subset \mathcal{P}_1(T)$ be the barycentric coordinates to the vertices $\{\mathbf{x}_i\}_{i=1}^3$ in $T\in \Th$ and define
\begin{equation*}
\omega_E^T:=\frac{1}{2}\cot{\theta_E^T}=-\int_T\nabla\lambda_i\cdot\nabla\lambda_j\;{\rm d}\mathbf{x}\quad\text{when}\quad\bm{\tau}_E=\mathbf{x}_j-\mathbf{x}_i.
\end{equation*}
Then, the mass lumping for the local mass matrix in $\mathcal{N}_0(T)$~\cite{haugazeau1993condensation} is
\begin{equation}
\int_T\mathbf{v}\cdot\mathbf{w}\;{\rm d}\mathbf{x}\approx\sum_{E\subset \partial T}\omega_E^T\left.(\mathbf{v}\cdot\bm{\tau}_E)\right|_E\left.(\mathbf{w}\cdot\bm{\tau}_E)\right|_E,\quad\forall\bv,\bw\in \mathcal{N}_0(T).\label{eqn: mass lumping approx}
\end{equation}
That is, we use the diagonal matrix ${\rm diag} ( \omega_{E_1}^T,  \omega_{E_2}^T, \omega_{E_3}^T)$ to approximate the $3\times 3$ dense matrix $(\int_T\bm{\chi}_{E_i}\cdot \bm{\chi}_{E_j}\;{\rm d}\mathbf{x})$. 
The mass lumping is exact when $\mathbf{v},\mathbf{w}\in \nabla \mathcal{P}_1(T)\subset \mathcal N_0(T)$. 
Consequently, we have
\begin{align*}
B^{T}_\mathcal{N}(u_h,v_h) &= \int_T J_T(u_h)\cdot\nabla v_h\;{\rm d}\mathbf{x}\\ 
&\approx \sum_{E\subset \partial T}\omega_E^T\bar{\kappa}_E\delta_E(e^{\psi_E} u_h)\delta_E(v_h)=:B_{h}^T(u_h,v_h).
\end{align*}
The EAFE discrete problem is to find $u_h\in V_h$ such that 
\begin{equation}\label{eqn: EAFE_discrete problem}
\sum_{T\in\mathcal{T}_h}B_h^T(u_h,v_h)=F(v_h),\quad\forall v_h\in V_h.
\end{equation}

In the EAFE bilinear form, it is clear to see $\bar{\kappa}_E>0$ and $\delta_E(e^{\psi_E} \lambda_i)\delta_E(\lambda_j)=-e^{\psi_E(\mathbf{x}_i)}<0$. Therefore, the off-diagonal entries in the stiffness matrix of the EAFE bilinear form in~\eqref{eqn: EAFE_discrete problem} are nonpositive if and only if $\omega_E^T+\omega_E^{T'}\geq 0$ for any interior edge $E=\partial T\cap \partial T'$, which is also equivalent to the Delaunay condition of $\Th$ in two dimensions.
Hence, the stiffness matrix corresponding to the EAFE bilinear form is an $M$-matrix if and only if the stiffness matrix for the Poisson equation (computed by $-\omega_E^T$) is an $M$-matrix.

In implementation, the coefficients $\alpha$ and $\bm{\beta}$ are approximated as a constant $\alpha_E$ and a constant vector $\bm{\beta}_E$ on each edge $E\subset\partial T$, respectively.
For example, $\alpha_E=(\alpha(\mathbf{x}_i)+\alpha(\mathbf{x}_j))/2$ and $\bm{\beta}_E=(\bm{\beta}(\mathbf{x}_i)+\bm{\beta}(\mathbf{x}_j))/2$.
The edgewise potential function is given as $\psi_E(\mathbf{x})={\alpha_E}^{-1}\bm{\beta}_E\cdot\mathbf{x}$.
Thus, the DoFs are explicitly expressed as
\begin{equation}\label{eqn: easy_computing}
\bar{\kappa}_E\delta_E(e^{\psi_E} u_h)=\alpha_E \mathbb{B}({\alpha_E}^{-1}\bm{\beta}_E\cdot(\mathbf{x}_{i}-\mathbf{x}_{j}))u_h(\mathbf{x}_j)-\alpha_E \mathbb{B}({\alpha_E}^{-1}\bm{\beta}_E\cdot(\mathbf{x}_{j}-\mathbf{x}_{i}))u_h(\bx_{i}),
\end{equation}
where $\mathbb{B}(z)$ is the Bernoulli function,
\begin{equation}
\mathbb{B}(z)=
\left\{
\begin{array}{cl}
\displaystyle\frac{z}{e^z-1}&  z\not=0,\\
1& z=0.
\end{array}
\right.
\label{bernoulli function}
\end{equation}


\section{Virtual element methods}\label{sec:VEM}
This section introduces the essential definitions and fundamental theories of the VEMs.
Let $\mathcal{K}_h$ denote a decomposition of $\Omega$ into polygonal elements $K$ with a mesh size
$h=\max_{K\in\mathcal{K}_h}h_K$, where $h_K$ is the diameter of $K$.
In a polygonal element $K$, $\mathbf{x}_i$ for $1\leq i\leq N_V$ denote vertices and $E_j$ for $1\leq j\leq N_E$ denote edges, where $N_V$ is the number of vertices and $N_E$ is the number of edges.
For a simple polygon, $N_V=N_E$. We label the indices of the vertices or edges counterclockwise.
We denote scaled tangent vectors on the edges by $\bm{\tau}_j=\mathbf{x}_{j+1}-\mathbf{x}_j$ for $1\leq j\leq N_E-1$, and $\bm{\tau}_{N_E}=\mathbf{x}_{1}-\mathbf{x}_{N_V}$.


\subsection{The lowest order nodal and edge virtual spaces}
We first introduce the lowest order local nodal space in \cite{BEI13,BEI14}.
Let the boundary space be
\begin{equation*}
B_1(\partial K):=\left\{ v_h\in C^0(\partial K):v_h|_E\in \mathcal{P}_1(E),\ \forall E\subset \partial K\right\}.
\end{equation*}
Then, the local nodal space~\cite{BEI13,BEI14} is defined as
\begin{equation*}
\mathcal{V}_1(K):=\left\{ v_h\in H^1(K): v_h|_{\partial K}\in B_1(\partial K),\ \Delta v_h|_K=0\right\},
\end{equation*}
and its DoFs are given as the values of $v_h$ at the vertices of $K$, i.e.,
\begin{equation}\label{vertexdof}
\text{dof}_i(v_h)=v_h(\mathbf{x}_i), \quad 1\leq i\leq N_V. 
\end{equation}
The dimension of $\mathcal V_1(K)$ is $N_V$, and
its canonical basis, $\{\phi_j\}_{j=1}^{N_V}\subset \mathcal V_1(K)$, is defined by $\text{dof}_i(\phi_j)=\delta_{ij}$.
It has been proved in \cite{BEI13} that the DoFs~\eqref{vertexdof} are unisolvent for $\mathcal V_1(K)$.
We also highlight that $\mathcal{P}_1(K)\subseteq \mathcal V_1(K)$, and the local $H^1$-projection is computable using the DoFs~\eqref{vertexdof}, that is, $\Pi^\nabla_{1}:\mathcal V_1(K)\rightarrow\mathcal{P}_1(K)$ with
\begin{align*}
\int_K\left(\nabla \Pi_{1}^\nabla v_h-\nabla v_h\right)\cdot \nabla p_1\;{\rm d}\mathbf{x}&=0,\quad\forall p_1\in\mathcal{P}_1(K),\\
\frac{1}{N_V}\sum_{i=1}^{N_V}\left(\Pi_{1}^\nabla v_h(\mathbf{x}_i)-v_h(\mathbf{x}_i)\right)&=0.
\end{align*}
The global nodal space is defined as
\begin{equation*}
\mathcal V_1(\mathcal{K}_h):=\left\{v_h\in H_0^1(\Omega):\left.v_h\right|_K\in \mathcal V_1(K),\ \forall K\in \mathcal{K}_h\right\},
\end{equation*}
and its global DoFs are all the values of $v_h$ at the internal vertices of $\mathcal{K}_h$.

We introduce the lowest order edge space with reduced DoFs~\cite{BEI16(3),BEI17(2)},
\begin{align*}
\mathcal{N}_0(K)=\left\{\mathbf{v}_h\in(L^2(K))^2 \right.: &\text{div}\;\mathbf{v}_h=0,\ \text{rot}\;\mathbf{v}_h\in\mathcal{P}_0(K),\\
&\left.\mathbf{v}_h\cdot\bm{\tau}_E\in\mathcal{P}_0(E),\ \forall E\subset\partial K\right\},
\end{align*}
where $\text{rot}\;\bv = \partial v_2/\partial x-\partial v_1/\partial y$
for a vector $\bv=\langle v_1,v_2\rangle$.
The DoFs for $\mathcal{N}_0(K)$~\cite{BEI16(3),BEI17(1),BEI17(2),BEI18} use the lowest order edge moments, similar to those of $\mathcal{N}_0(T)$,
\begin{equation}
\text{dof}_E(\mathbf{v}_h)=\frac{1}{|E|}\int_{E}\mathbf{v}_h\cdot\bm{\tau}_{E}\;{\rm d}s=\left.(\mathbf{v}_h\cdot \bm{\tau}_E)\right|_E\label{edgedof},
\end{equation}
for each edge $E\subset \partial K$, and
the canonical basis, $\{ \bm{\chi}_j\}_{j=1}^{N_E}\subset \mathcal{N}_0(K)$, is defined by $\text{dof}_{E_i}(\bm{\chi}_j)=\delta_{ij}$ for $1\leq i\leq N_E$. Note that, unlike the finite element spaces, explicit forms of $\{ \bm{\chi}_j\}_{j=1}^{N_E}\subset \mathcal{N}_0(K)$ are unknown. Instead, only the DoFs are used to construct a discretization.

By integration by parts, $\text{rot}\;\mathbf{v}_h$ is determined by the DoFs~\eqref{edgedof}, so the dimension of $\mathcal{N}_0(K)$ is $N_E$.
It is easy to verify the unisolvence for $\mathcal{N}_0(K)$ because the divergence-free condition implies $\mathbf{v}_h=\textbf{rot}\;\vartheta=\langle\partial\vartheta/\partial y,-\partial\vartheta/\partial x\rangle$ for some differentiable function $\vartheta$, and integration by parts leads to
\begin{equation*}
\int_K|\mathbf{v}_h|^2\;{\rm d}\mathbf{x}=\int_K\mathbf{v}_h\cdot\textbf{rot}\;\vartheta\;{\rm d}\mathbf{x}=\int_K(\text{rot}\;\mathbf{v}_h)\vartheta\;{\rm d}\mathbf{x}-\sum_{E\subset\partial K}\frac{1}{|E|}\int_E(\mathbf{v}_h\cdot\bm{\tau}_E)\vartheta\;{\rm d}s.
\end{equation*}
We also emphasize that $\mathcal{N}_0(T)\subset \mathcal{N}_0(K)$, and the local $L^2$-projection $\bm{\Pi}_{0}^0:\mathcal{N}_0(K)\rightarrow(\mathcal{P}_0(K))^2$ is computable by the DoFs~\eqref{edgedof},
\begin{equation*}
\int_K\mathbf{v}_h\cdot\mathbf{p}_0\;{\rm d}\mathbf{x}=\int_K\mathbf{v}_h\cdot\textbf{rot}\;p_1\;{\rm d}\mathbf{x}=\int_K(\text{rot}\;\mathbf{v}_h)p_1\;{\rm d}\mathbf{x}-\sum_{E\subset\partial K}\frac{1}{|E|}\int_E(\mathbf{v}_h\cdot\bm{\tau}_E)p_1\;{\rm d}s,
\end{equation*}
for some $p_1\in \mathcal{P}_1(K)$ with $\textbf{rot}\;p_1 = \mathbf{p}_0$.

The global edge virtual element space is required to preserve the $H(\textbf{curl})$ conformity~\cite{BEI17(1),BEI17(2),BEI18} with respect to the space
\begin{equation*}
H_0(\text{rot};\Omega)=\{\bv\in(L^2(\Omega))^2: \text{rot}\;\bv\in L^2(\Omega),\ \bv\cdot\mathbf{t}=0\ \text{on}\ \partial \Omega\},
\end{equation*}
where $\mathbf{t}$ denotes the unit tangent vector.
Hence, the two-dimensional global space is defined as
\begin{equation*}
\mathcal{N}_0(\mathcal{K}_h):=\left\{\mathbf{v}_h\in H_0(\text{rot};\Omega):\left.\mathbf{v}_h\right|_K\in \mathcal{N}_0(K),\ \forall K\in \mathcal{K}_h\right\},
\end{equation*}
The global DoFs are $\mathbf{v}_h\cdot\bm{\tau}_E$ for all interior edges of the decomposition $\mathcal{K}_h$.

We present an important relation between the local nodal space $\mathcal V_1(K)$ and the local edge space $\mathcal{N}_0(K)$.
The detailed proof can be found in \cite{BEI18}.
\begin{proposition}\label{prop relation}
$\nabla \mathcal V_1(K)$ is a subset of $\mathcal{N}_0(K)$, and moreover
\begin{equation*}
\nabla \mathcal V_1(K) = \left\{\mathbf{v}_h\in \mathcal{N}_0(K)\mid \textnormal{rot}\; \mathbf{v}_h=0\right\}.
\end{equation*}
\end{proposition}

In addition, we show the relation between the canonical basis functions $\{\phi_i\}_{i=1}^{N_V}$ for $\mathcal V_1(K)$ and $\left\{ \bm{\chi}_j\right\}_{j=1}^{N_E}$ for $\mathcal{N}_0(K)$.
\begin{lemma} \label{lemma relation}
For the canonical basis $\bm{\chi}_m,\bm{\chi}_n\in \mathcal{N}_0(K)$ with $1\leq m<n\leq N_E$, we have
\begin{equation*}
\bm{\chi}_n-\bm{\chi}_m=-\sum_{l=m+1}^{n}\nabla\phi_l,
\end{equation*}
where $\phi_l$ is the canonical basis of $\mathcal{V}_1(K)$.
\end{lemma}
\begin{proof}
Let $\mathbf{w}_h=-\sum_{l=m+1}^{n}\nabla\phi_l$. Then, $\mathbf{w}_h\in \nabla \mathcal V_1(K)\subset\mathcal{N}_0(K)$ by Proposition~\ref{prop relation}.
It is clear to see that
$\text{dof}_{E}(\mathbf{w}_h)=\text{dof}_{E}(\bm{\chi}_n-\bm{\chi}_m)$,
for all $E\subset \partial K$. Therefore, $\mathbf{w}_h=\bm{\chi}_n-\bm{\chi}_m$ of $\mathcal{N}_0(K)$ by the unisolvent condition. 
\end{proof}


\subsection{Bilinear forms for the Poisson equation}
We introduce bilinear forms appearing in virtual element methods for the Poisson equation.
The local continuous bilinear form for the Poisson equation is denoted as
\begin{equation}\label{eqn: conti_bilinear}
a^K(u,v) = \int_K \nabla u\cdot\nabla v\;{\rm d}\mathbf{x}.
\end{equation}
Since we do not know the explicit forms of $u_h,v_h\in \mathcal V_1(K)$ inside $K$, it would not be possible to compute $a^K(u_h,v_h)$ exactly.
Therefore, we use the equivalent bilinear form~\cite{BEI13,BEI14},
\begin{equation}
a_h^K(u_h,v_h)=\int_K\nabla \Pi_1^\nabla u_h\cdot\nabla\Pi_1^\nabla v_h\;{\rm d}\mathbf{x}+S^K(u_h-\Pi_1^\nabla u_h,v_h-\Pi_1^\nabla v_h)\label{Poisson bilinear form}
\end{equation}
for all $u_h,v_h\in \mathcal V_1(K)$. In this case, the stabilization term $S^K$ needs to satisfy that there are positive constants $c_0$ and $c_1$ independent of $h_K$ such that
\begin{equation*}
c_0 a^K(v_h,v_h)\leq S^K(v_h,v_h)\leq c_1a^K(v_h,v_h),
\end{equation*}
for every $v_h\in \mathcal V_1(K)$ with $\Pi_1^\nabla v_h=0$.
These inequalities directly imply the norm equivalence,
\begin{equation}
\gamma_*a^K(v_h,v_h)\leq a_h^K(v_h,v_h)\leq \gamma^*a^K(v_h,v_h),\quad\forall v_h\in \mathcal V_1(K),\label{stability}
\end{equation}
for some positive constants $\gamma_*$ and $\gamma^*$ independent of $h_K$ (see \cite{BEI13,BEI14,CHE18} for details).
Here, we introduce the most popular choice for $S^K$,
\begin{equation}
S_\mathcal{V}^K(u_h,v_h):=\sum_{i=1}^{N_V}\text{dof}_i(u_h)\text{dof}_i(v_h).\label{popular stabilization}
\end{equation}
Another choice of the stabilization term is the $H^{\frac{1}{2}}$-inner product~\cite{CAO18},
\begin{align}
\label{H half stabilization}S^K_\mathcal{E}(u_h,v_h):
=\sum_{E\subset \partial K}\delta_E(u_h)\delta_E(v_h),
\end{align}
where $\delta_E$ is defined in~\eqref{def delta}.
The verification of norm equivalence~\eqref{stability} for~\eqref{popular stabilization} can be found in \cite{CHE18} and for~\eqref{H half stabilization} in \cite{CAO18}, respectively.

In Section~\ref{sec: fvm_vem}, we will present a discretization of VEMs based on the finite volume formulation requiring no stabilization. 


\subsection{Construction of the right-hand side}\label{subsec: right hand side}

We use the approximation technique~\cite{BEI13} for the right-hand side. Let $f_0^K$ be a piecewise constant approximation of $f$. Then, we define
\begin{equation*}
F_h(v_h) := \sum_{K\in\mathcal{K}_h}\int_K f_0^K\bar{v}_h\;{\rm d}\mathbf{x}=\sum_{K\in\mathcal{K}_h}|K|f_0^K\bar{v}_h,\quad \bar{v}_h:=\frac{1}{N_V}\sum_{i=1}^{N_V} v_h(\mathbf{x}_i).
\end{equation*}
This approximation yields an optimal estimate~\cite{BEI13},
\begin{equation}
F_h(v_h)-F(v_h)\leq Ch\left(\sum_{K\in\mathcal{T}_h}\left|f\right|_{H^1(K)}^2\right)^{1/2}\left|v_h\right|_{H^1(\Omega)}.\label{estimate of rhs}
\end{equation}


\section{A finite volume bilinear form with virtual element}\label{sec: fvm_vem}
In this section, we present a finite volume bilinear form for the Poisson equation corresponding to $a^K(\cdot,\cdot)$ in~\eqref{eqn: conti_bilinear} on Voronoi meshes with respect to dual Delaunay triangulations consisting of acute triangles.
We employ piecewise constant approximations on dual edge patches inspired by \cite{brezzi2006error} and show that the finite volume bilinear form is compatible with 
the virtual element space $\mathcal V_1(\mathcal{K}_h)$ and stabilization-free.


\subsection{Voronoi meshes with dual Delaunay triangulations}
In two dimensions, a Delaunay triangulation consists of triangles whose circumcircles include no other vertex of the triangulation, which means that the sum of the two angles opposite to any internal edge is less than or equal to $\pi$.
On the other hand, we call a partition a Voronoi tessellation if each seed point has a corresponding region of all vertices closer to that seed than any other seed using the Euclidean distance.
The Delaunay and Voronoi partitions are dual in a convex domain because there are one-to-one relationships between vertices on one mesh and elements on the other. Indeed, each seed point of a Voronoi element is a vertex of a Delaunay triangle, while each circumcenter of a Delaunay triangle is a vertex of a Voronoi polygon.
Moreover, the duality between Delaunay and Voronoi means that the edges on a mesh are orthogonal to the ones on the other.
This dual configuration plays an essential role in finite volume methods~\cite{eymard2000finite}.

In this section, we assume all the triangles determined by interior nodes are acute; each triangle includes its circumcenter. Then, we generate a Voronoi mesh by connecting the neighboring circumcenters; each Voronoi element contains a vertex of a Delaunay triangle.
We consider the Voronoi mesh primary for cell-centered finite volume methods, while the Delaunay triangulation is dual (see the left figure in Figure~\ref{DelVoro}). This primary mesh is considered a sub-class of general Voronoi meshes whose dual Delaunay may include obtuse triangles.

On the boundary of the domain, we obtain Voronoi nodes by extending perpendicular bisectors of triangles' edges around the boundary, where additional corner Voronoi nodes may help to express the boundary more accurately.
Boundary Delaunay nodes are located between two Voronoi nodes along the boundary, satisfying the orthogonality of the triangles' edges and the boundary.
Hence, the control volumes (or triangles) in the dual Delaunay triangulation can cover the whole domain.
However, since considering the zero boundary condition~\eqref{eqn: governing2}, we exclude the control volumes containing the boundary Deluanay nodes and employ the one-to-one relationships between a control volume and a primary (Voronoi) vertex (see the right figure in Figure~\ref{DelVoro}).

\begin{figure}[h!]
       \centering
        \scalebox{0.4}{
  \includegraphics[width=10cm]{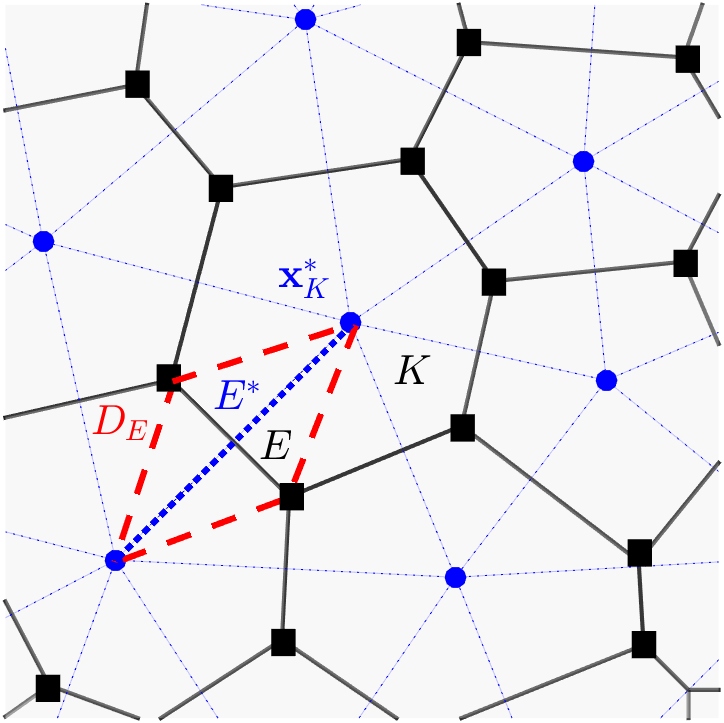}\hskip 1in
  \includegraphics[width=10cm]{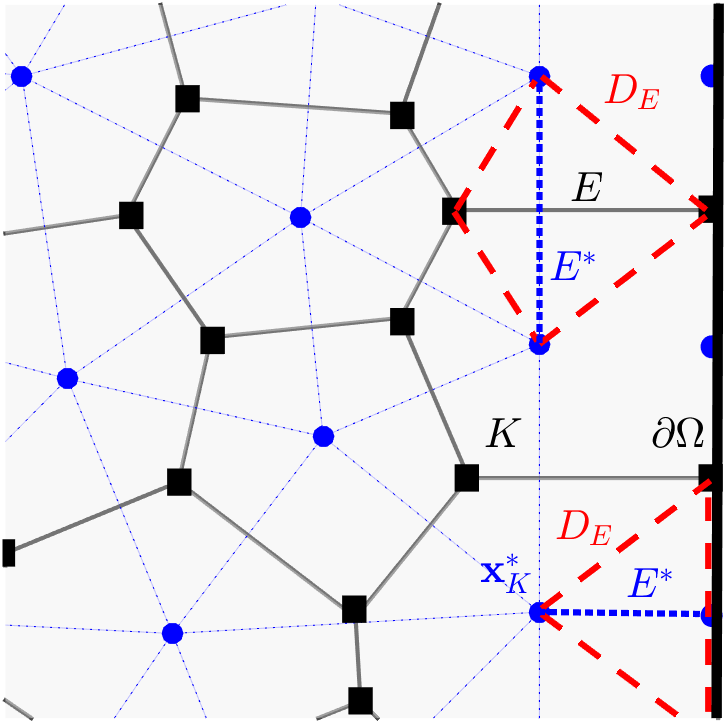}}
  \caption{A primary Voronoi mesh with a dual Delaunay triangulation consisting of acute triangles; black squares are Voronoi nodes, and blue circles are seed points (or Delaunay nodes).}\label{DelVoro}
  \end{figure}
  We let $\Kh$ denote a primary Voronoi mesh with a dual Delaunay triangulation $\mathcal{T}_h=\mathcal{K}_h^*$ consisting of acute triangles with interior nodes.
Each polygon $K\in\Kh$ contains only one Delaunay vertex $\mathbf{x}_K^*$ in the interior while each interior Voronoi vertex is the circumcenter of a triangle $T\in\mathcal{T}_h$.
Then, an edge $E$ of $K$ and $\mathbf{x}_K^*$ determine a sub-triangle in $K$.
Any polygon $K$ is split into sub-triangles, and the number of sub-triangles is the number of edges of $K$.
Every internal edge in $\Kh$ is shared by two such sub-triangles.
We let $D_E$ denote the union of the two sub-triangles, and $E^*$ denote the dual edge in $\Th$ included in $D_E$ with $E^*\perp E$.
For a boundary edge $E\subset\partial\Omega$, $D_E$ is defined as one triangle appearing in the domain $\Omega$, and $E^*$ is defined as the line segment from $\bx_K^{*}$ to $E$. Then, all the dual edge patches $D_E$ form a polygonal mesh of $\Omega$ and will be denoted by $\mathcal{D}_h$ (see Figure~\ref{DelVoro} for an illustration). 
Moreover, $\Eho$ denotes the collection of all interior edges in $\mathcal{K}_h$.
We also note that if $\bn_{E^*}$ is a unit normal vector on $E^*$ and $\bm{\tau}_E$ is a scaled tangent vector on $E$, then 
$\bn_{E^*}=\bm{\tau}_E/|E|$ 
because $E^*\perp E$. By the construction, their intersection 
$\mathbf{m}_{E^*} = E^*\cap E$
is the midpoint of $E^*$ but may not be the midpoint of $E$.

Finally, we recall the virtual element spaces on the primary Voronoi mesh and define function spaces on the dual triangulation and patches. 
\begin{itemize}
    \item $\mathcal{V}_1(\mathcal{K}_h)\subset H_0^1(\Omega)$: the linear nodal virtual element space on the polygons $K\in\mathcal{K}_h$.
    \item $\mathcal{N}_0(\mathcal{K}_h)\subset H_0(\text{rot};\Omega)$: the lowest order edge virtual element space on the polygons $K\in\mathcal{K}_h$.
    \item $V_0(\mathcal{T}_h)$: the set of piecewise constant functions $v_h^*$ on the triangles $T\in\Th$, excluding the control volumes containing the boundary Delaunay nodes.
    \item $W_0(\mathcal{D}_h)$: the set of piecewise constant functions $w_h$ on the edge patches $D_E\in\mathcal{D}_h$, where $w_h|_{D_E}\equiv0$ if $E\subset\partial \Omega$.
\end{itemize}


\subsection{A finite volume bilinear form on Voronoi meshes}\label{subsec: fvm_vespace}
We define a one-to-one map between $\mathcal{V}_1(\mathcal K_h)$ and $V_0(\Th)$ by the DoFs. In each interior triangle $T\in\Th$, there is only one interior Voronoi vertex $\bx_T$ because of the duality between $\mathcal{K}_h$ and $\Th$, so the DoFs of $v_h\in \mathcal V_1(\mathcal{K}_h)$ at the vertex, $v_h(\bx_T)$, maps to $v_h^*|_T$ of $v_h^*\in V_0(\Th)$. Through this one-to-one mapping, we can extend the definition of $\delta_E$ to functions in $V_0(\Th)$, i.e., $\delta_E(v_h^*) = \delta_E (v_h)$ is the difference of function values of $v_h$ on the vertices of $E$ following the orientation of $E$.

We consider a balance equation for the Poisson equation over each $T\in \Th$,
\begin{equation}
-\int_{\partial {T}}\nabla u\cdot \bn\;{\rm d}s=\int_{T}f\;{\rm d}\bx.\label{eqn: balance_eqn}
\end{equation}
We define a normal flux approximation $\mathcal{G}_{\bn}(u_h)\in W_0(\mathcal D_h)$ as 
\begin{equation}
\left.\mathcal{G}_{\bn}(u_h)\right|_{D_E}:=\nabla u_h(\mathbf{m}_{E^*})\cdot \bn_{E^*} =\nabla u_h(\mathbf{m}_{E^*})\cdot \frac{\bm{\tau}_{E}}{|E|} =\frac{\delta_E(u_h)}{|E|}.\label{eqn: def_discrete_grad}
\end{equation}
A cell-centered finite volume method uses the discrete normal flux on $E^*\subset D_E$
\begin{equation*}
\int_{E^*}\nabla u_h \cdot \bn_{E^*}\;{\rm d}s\approx \int_{E^*}\mathcal{G}_{\bn}(u_h)\;{\rm d}s=\frac{|E^*|}{|E|}\delta_E(u_h).
\end{equation*}
For $u_h$ in the virtual element space $\mathcal V_1(\mathcal{K}_h)$, $\nabla u_h \cdot \bn_{E^*}$ is not computable but can be evaluated at the midpoint $\mathbf{m}_{E^*}$. By the duality, it equals to the tangential derivative along $E$, which is computable as $u_h|_E$ is linear. 
Then, if $\phi_T^{*}\in V_0(\Th)$ is the characteristic function of a triangle $T\in \Th$, i.e. $\phi_T^{*}\equiv 1$ in $T$ and $\phi_T^{*}\equiv 0$ in $\Omega\setminus T$, the balance equation~\eqref{eqn: balance_eqn} implies a discrete problem
\begin{equation*}
\sum_{E\in\Eho}\frac{|E^*|}{|E|}\delta_E(u_h)\delta_E(\phi_T^{*})=\int_{T}f\;{\rm d}\bx=\int_{\Omega}f\phi_T^{*}\;{\rm d}\bx=F(\phi_T^{*}).
\end{equation*}

By using linear combinations in $V_0(\Th)$, we define a local bilinear form
\begin{equation*}
\mathbf{a}_h^K(u_h, v_h^{*}):=\sum_{E\subset \partial K}\omega_E^K\delta_E(u_h)\delta_E(v_h^{*}),\quad \forall v_h^{*}\in V_0(\mathcal{T}_h),
\end{equation*}
where $\omega_E^K:=|E^*|/(2|E|)$.
Therefore, the corresponding discrete problem is to find $u_h\in \mathcal V_1(\Kh)$ such that
\begin{equation}
\sum_{K\in\Kh}\mathbf{a}_h^K(u_h, v_h^{*})=F(v_h^{*}),\quad \forall v_h^{*}\in V_0(\mathcal{T}_h),\label{eqn: fvm_discrete_prob}
\end{equation}
which is a Petrov-Galerkin formulation of the Poisson equation. By changing $\delta_E(v_h^{*})$ to $\delta_E(v_h)$, we obtain a Galerkin formulation $\mathbf{a}_h^K(u_h, v_h)$.

In the virtual element framework, the bilinear form can be explained in terms of mass lumping in the edge virtual element space $\mathcal{N}_0(K)$ for $K\in \Kh$ (see also~\eqref{eqn: mass lumping approx}),
\begin{equation}
a^K(u_h,v_h)=\int_K\nabla u_h\cdot\nabla v_h\;{\rm d}\bx\approx \sum_{E\subset \partial K}\omega_E^K\left.(\nabla u_h\cdot\bm{\tau}_E)\right|_E\left.(\nabla v_h\cdot\bm{\tau}_E)\right|_E=\mathbf{a}_h^K(u_h, v_h)\label{eqn: vem_mass_lump}
\end{equation}
because $\left.(\nabla v_h\cdot\bm{\tau}_E)\right|_E=\delta_E(v_h)$.
Moreover, we emphasize that the DoFs of $v_h\in \mathcal V_1(\Kh)$ compute the bilinear form $\mathbf{a}_h^K(\cdot,\cdot)$ without any stabilization term, so the discretization~\eqref{eqn: fvm_discrete_prob} can be viewed as a stabilization-free virtual element method.

For the right-hand side, we apply the mass lumping with the vertices of a triangle $T\in \Th$ by using the constant approximation $f_0^K$ on $K$ introduced in Section~\ref{subsec: right hand side}.
More precisely, when the three vertices of $T_j\in \Th$ are in polygons $K_1$, $K_2$, and $K_3$ in $\Kh$, the right hand side is calculated by
\begin{equation}
F(\phi_j^{*})=\int_{\Omega} f\phi_j^{*}\;{\rm d}\bx\approx\frac{|T_j|}{3}\left(f_0^{K_1}+f_0^{K_2}+f_0^{K_3}\right)=:\mathcal{F}_h(\phi_j^{*}).\label{eqn: fvm_rhs}
\end{equation}
Moreover, if the polygons $K_1$, $K_2$, and $K_3$ are regular, then $\mathcal{F}_h(\phi_j^*)=F_h(\phi_j)$ in Section~\ref{subsec: right hand side}.


\subsection{$M$-matrix property}
We stress that the stiffness matrix corresponding to $\ba_h^K(\cdot,\cdot)$ is an $M$-matrix through the following lemma.
\begin{lemma}\label{lemma: fvm_m_matrix}
If $\Kh$ is a primary Voronoi mesh with a dual Delaunay triangulation consisting of acute triangles, then the stiffness matrix corresponding to $\ba_h^K(\cdot,\cdot)$ is an $M$-matrix.
\end{lemma}
\begin{proof}
Let $\zeta_j$ be a global nodal virtual function in $\mathcal V_1(\Kh)$ such that $\zeta_j(\mathbf{x}_i)=\delta_{ij}$ for any internal vertex $\mathbf{x}_i$ in $\Kh$.
We consider the stiffness matrix whose entries are
\begin{equation*}
\mathbf{A}_{ij}=\sum_{K\in\Kh}\ba_h^{K}(\zeta_j,\zeta_i).
\end{equation*}
If $\mathbf{x}_i$ and $\mathbf{x}_j$ ($i\not=j$) are the endpoints of an edge $E\in \Eho$, we have
\begin{equation*}
\ba_{h}^K(\zeta_j,\zeta_i)=-\omega_E^K=-|E^*|/(2|E|)<0.
\end{equation*}
The partition of unity implies that
\begin{equation*}
\sum_{i}\mathbf{A}_{ij}=\sum_{K\in\mathcal{K}_{h}^{}}\mathbf{a}_h^K(\zeta_j,1)=0
\end{equation*}
if $\mathbf{x}_j$ has no neighbor on the boundary.
If $\mathbf{x}_j$ has a neighbor on the boundary, then $\sum_i \mathbf{A}_{ij}>0$.
Hence, the stiffness matrix from the bilinear form $\mathbf{a}_{h}^K(\cdot,\cdot)$ is an invertible $M$-matrix.
\end{proof}


\subsection{Convergence analysis for Poisson equation}
We use an alternative form with the normal flux approximation~\eqref{eqn: def_discrete_grad} on the edge patches $D_E$ for convergence analysis,
\begin{equation}
\sum_{K\in\Kh}\mathbf{a}_h^K(u_h, v_h)=\sum_{E\in \Eho}\frac{|E^*|}{|E|}\delta_E(u_h)\delta_E(v_h)=2\sum_{E\in \Eho}\int_{D_E}\mathcal{G}_{\bn}(u_h)\mathcal{G}_{\bn}(v_h)\;{\rm d}\mathbf{x}\label{eqn: fvm_alt_form}
\end{equation}
for $u_h,v_h\in \mathcal V_1(\Kh)$ because $|E^*||E|=2|D_E|$.

\begin{theorem}\label{thm: fvm_conver_analy}
Assume that $\mathcal{K}_h$ is a primary Voronoi mesh with a dual Delaunay triangulation consisting of acute triangles, $\nabla u\cdot \bn_{E^*}\in H^1(D_E)$ for all $E\in \Eho$, and $u\in C^0(\bar{\Omega})$. Then, we have for $u_h\in \mathcal V_1(\Kh)$ and $\mathcal{G}_\bn(u_h)$ in \eqref{eqn: def_discrete_grad},
\begin{equation*}
\left(\sum_{E\in\Eho}\norm{\nabla u \cdot {\mathbf n}_{E^*}-\mathcal{G}_{\bn}(u_h)}^2_{L^2(D_E)}\right)^{1/2}\leq Ch\left(\sum_{E\in\Eho}\snorm{\nabla u\cdot {\mathbf n}_{E^*}}^2_{H^1(D_E)}\right)^{1/2}.
\end{equation*}
\end{theorem}

\begin{proof}
We adopt the idea of the analysis in \cite{brezzi2006error}.
We define two different approximations of the exact normal flux $\nabla u\cdot \bn_{E^*}$.
The first one, $\bar{\mathcal{G}}_{\bn}(u)\in W_0(\mathcal{D}_h)$, is the average of the flux on $E^*\subset D_E$,
\begin{equation*}
\left.\bar{\mathcal{G}}_{\bn}(u)\right|_{D_E}:=\frac{1}{|E^*|}\int_{E^*} \nabla u\cdot\bn_{E^*}\;{\rm d}s
\end{equation*}
for each $D_E\in \mathcal{D}_h$. 
Then, it follows from the balance equation~\eqref{eqn: balance_eqn}, the definitions of $\bar{\mathcal{G}}_{\bn}(u)$ and $\mathcal{G}_{\bn}(v_h)$, and the fact $|E^*||E|=2|D_E|$ that
\begin{align*}
2\sum_{E\in\Eho}\int_{D_E}\bar{\mathcal{G}}_{\bn}(u)\mathcal{G}_{\bn}(v_h)\;{\rm d}\mathbf{x}=\sum_{E\in \Eho}|E^*|(\bar{\mathcal{G}}_{\bn}(u)|_{D_E})\delta_E(v_h^{*})=F(v_h^*).
\end{align*}
Hence, if we compare this equation with the discrete problem~\eqref{eqn: fvm_discrete_prob} applying~\eqref{eqn: fvm_alt_form}, then we have the orthogonality
\begin{equation*}
\sum_{E\in \Eho}\int_{D_E} (\mathcal{G}_{\bn}(u_h)-\bar{\mathcal{G}}_{\bn}(u)) \mathcal{G}_{\bn}(v_h)\;{\rm d}\mathbf{x}=0, \quad \forall v_h\in \mathcal V_1(\mathcal{K}_h).
\end{equation*}
The second approximation, $u_I\in \mathcal V_1(\Kh)$, is the nodal value interpolant of $u$ onto the virtual element space.
Then, we obtain
\begin{align*}
\sum_{E\in\Eho}\norm{\mathcal{G}_{\bn}(u_h)-\bar{\mathcal{G}}_{\bn}(u)}&^2_{L^2(D_E)}=\sum_{E\in\Eho}\int_{D_E}(\mathcal{G}_{\bn}(u_h)-\bar{\mathcal{G}}_{\bn}(u)) (\mathcal{G}_{\bn}(u_I)-\bar{\mathcal{G}}_{\bn}(u))\;{\rm d}\bx\\
&\leq\sum_{E\in\Eho}\norm{\mathcal{G}_{\bn}(u_h)-\bar{\mathcal{G}}_{\bn}(u)}_{L^2(D_E)}\norm{\mathcal{G}_{\bn}(u_I)-\bar{\mathcal{G}}_{\bn}(u)}_{L^2(D_E)}.
\end{align*}
Moreover, the Cauchy-Schwarz inequality implies that
\begin{equation*}
\sum_{E\in\Eho}\norm{\mathcal{G}_{\bn}(u_h)-\bar{\mathcal{G}}_{\bn}(u)}^2_{L^2(D_E)}\leq\sum_{E\in\Eho}\norm{\mathcal{G}_{\bn}(u_I)-\bar{\mathcal{G}}_{\bn}(u)}^2_{L^2(D_E)}.
\end{equation*}
Using the triangle inequality, we have
\begin{equation*}
\norm{\mathcal{G}_{\bn}(u_I)-\bar{\mathcal{G}}_{\bn}(u)}_{L^2(D_E)}\leq \norm{\nabla u\cdot\bn_{E^*}-\mathcal{G}_{\bn}(u_I)}_{L^2(D_E)}+ \norm{\nabla u\cdot \bn_{E^*}-\bar{\mathcal{G}}_{\bn}(u)}_{L^2(D_E)}.
\end{equation*}
Each term means the first order approximation because $\mathcal{G}_{\bn}(u_I)$ equals $\nabla u\cdot\bn_{E^*}$ at a point on $E\subset D_E$, and $\bar{\mathcal{G}}_{\bn}(u)$ equals $\nabla u\cdot\bn_{E^*}$ at a point on $E^*\subset D_E$.
\end{proof}


\section{A monotone virtual element method}\label{sec: MEAVE}

We present a monotone virtual element scheme by utilizing the finite volume approach in Section~\ref{sec: fvm_vem}.


\subsection{A monotone edge-averaged virtual element method}

Let $\Kh$ be a primary Voronoi mesh with a dual Delaunay triangulation $\Th=\mathcal{K}_h^*$ consisting of acute triangles (see Figure~\ref{DelVoro}).
We consider the same one-to-one map between $\mathcal V_1(\mathcal{K}_h)$ and $V_0(\Th)$, and the constant space $W_0(\mathcal{D}_h)$ on the edge patches $D_E\in\mathcal{D}_h$ as in Section~\ref{sec: fvm_vem}.
For $J_K(u_h)\in\mathcal{N}_0(K)$, $J_K(u_h)\cdot\bn_{E^*}=(J_K(u_h)\cdot\bm{\tau}_E)/|E|$ on each interior edge $E\in\Eho$.
Thus, we define a flux approximation $\mathcal{J}_{\bn}(u_h)\in W_0(\mathcal{D}_h)$ in light of~\eqref{def delta},
\begin{equation}
\left.\mathcal{J}_{\bn}(u_h)\right|_{D_E}:= J_K(u_h)(\mathbf{m}_{E^*})\cdot \bn_{E^*} =J_K(u_h)(\mathbf{m}_{E^*})\cdot \frac{\bm{\tau}_E}{|E|}= \frac{\bar{\kappa}_E\delta_E(e^{\psi_E}u_h)}{|E|}.\label{eqn: definition of J and G}
\end{equation}
The normal flux $J_K(u_h)\cdot \bn_{E^*}$ is computable at the intersection point $\mathbf{m}_{E^*}$. 
In a balance equation corresponding to the convection-diffusion equation~\eqref{eqn: governing1} over each $T\in \Th$,
\begin{equation}
-\int_{\partial T}J(u)\cdot \bn\;{\rm d}s=\int_{T}f\;{\rm d}\bx,\label{eqn: conservative_formulation}
\end{equation}
the discrete flux $\mathcal{J}_{\bn}(u_h)$ approximates the flux $J(u_h)\cdot\bn_{E^*}$ on each $E^*\subset D_E$,
\begin{equation*}
\int_{E^*} J(u_h)\cdot \bn_{E^*}\;{\rm d}s\approx \int_{E^*}\mathcal{J}_{\bn}(u_h)\;{\rm d}s=\frac{|E^*|}{|E|}\bar{\kappa}_E\delta_E(e^{\psi_E}u_h).
\end{equation*}
By introducing $\phi_T^{*}\in V_0(\Th)$ such that $\phi_T^{*}\equiv 1$ in $T$ and $\phi_T^{*}\equiv 0$ in $\Omega\setminus T$, a cell-centered finite volume method for the balance equation~\eqref{eqn: conservative_formulation} is represented as
\begin{equation}
\sum_{E\in\Eho}\frac{|E^*|}{|E|}\bar{\kappa}_E\delta_E(e^{\psi_E}u_h)\delta_E(\phi_T^{*})=\int_{T}f\;{\rm d}\bx=F(\phi_T^{*}).\label{eqn: cell_fvm_local}
\end{equation}
Thus, using
linear combinations of such $\phi_T^{*}$ in $V_0(\Th)$,
we define a local bilinear form on $K\in\Kh$,
\begin{equation*}
\mathcal{B}_h^K(u_h, v_h^*):=\sum_{E\subset \partial K}\omega_E^K\bar{\kappa}_E\delta_E(e^{\psi_E}u_h)\delta_E(v_h^*),
\end{equation*}
where $\omega_E^K=|E^*|/(2|E|)$.
Therefore, the corresponding discrete problem is to seek $u_h\in \mathcal V_1(\Kh)$ such that
\begin{equation}
\sum_{K\in\Kh}\mathcal{B}_{h}^K(u_h,v_h ^*)=F(v_h^{*}),\quad\forall v_h^*\in V_0(\mathcal{T}_h^{}),\label{eqn: discrete problem_MEAVE}
\end{equation}
which is a Petrov-Galerkin discretization for the convection-diffusion equation~\eqref{eqn: governing1}.

We can also apply the mass lumping to derive a Galerkin formulation 
\begin{equation*}
\int_K J_K(u_h)\cdot\nabla v_h\;{\rm d}\bx\approx \sum_{E\subset K}\omega_E^K \left.(J_K(u_h)\cdot\bm{\tau}_E)\right|_E\left.(\nabla v_h\cdot\bm{\tau}_E)\right|_E=:\mathcal{B}_h^K(u_h,v_h).
\end{equation*}
This approach can be considered as a cell-centered finite volume Scharfetter-Gummel method.
One can find a vertex-centered finite volume Scharfetter-Gummel method~\cite{bank1998finite} when the primary mesh is a Delaunay triangulation and the dual one is a Voronoi mesh.
The vertex-centered method is identical to the EAFE method~\cite{LAZ05,XU99}.
Also, a stable mimetic finite difference (MFD) method~\cite{adler2022stable} is considered as a vertex-centered method.

\begin{corollary}\label{lemma: MEAVE_monotone}
Provided that a polygonal mesh $\Kh$ is a Voronoi mesh with a dual Delaunay triangulation consisting of acute triangles.
For $\alpha\in C^0(\bar{\Omega})$ and $\bm{\beta}\in(C^0(\bar{\Omega}))^2$, the EAVE method~\eqref{eqn: discrete problem_MEAVE} is monotone, and consequently~\eqref{eqn: discrete problem_MEAVE} is well-posed.
\end{corollary}
\begin{proof}
The proof is similar to that of Lemma~\ref{lemma: fvm_m_matrix}, so we omit the proof.
\end{proof}


\section{General framework for EAVE methods}\label{sec:EAVE}

This section presents a general framework for edge-averaged virtual element (EAVE) methods.
In this framework, we derive another EAVE bilinear form working with general polygonal meshes.
The monotone EAVE method in Section~\ref{sec: MEAVE} is a particular case of this general framework by applying the mass lumping in $\mathcal{N}_0(K)$ on special polygonal meshes.

\subsection{EAVE method on general polygonal meshes}
We begin with the bilinear form with the flux approximation $J_K(u_h)\in\mathcal{N}_0(K)$ in light of~\eqref{def delta},
\begin{equation}
\int_K J_K(u_h)\cdot\nabla v_h\;{\rm d}\mathbf{x}=\sum_{E\subset \partial K}\bar{\kappa}_E\delta_E(e^{\psi_E}u_h)\int_K \bm{\chi}_E\cdot\nabla v_h\;{\rm d}\mathbf{x}.\label{bilinear form Bbar}
\end{equation}
We then apply a mass lumping technique and Lemma~\ref{lemma relation} (see \cite{lee2021edge} for details). 
Let $E_{ij}$ be the line segment joining $\mathbf{x}_i$ and $\mathbf{x}_j$ for $i<j$ in $K$.
Then, the EAVE bilinear form in a general polygon $K$ is defined as
\begin{equation}
B_h^K(u_h,v_h) := \sum_{1\leq i<j\leq N_V} \omega_{ij}^K\bar{\kappa}_{E_{ij}}\delta_{ij}^K(e^{\psi_{E_{ij}}} u_h)\delta_{ij}^K(v_h),\label{bilinear form Bh}
\end{equation}
where
$\omega_{ij}^K=-a_h^K(\phi_i,\phi_j)$ in~\eqref{Poisson bilinear form}, $\bar{\kappa}_{E_{ij}}$ is the harmonic average of $\kappa_{E_{ij}}$, and $\delta_{ij}^K(v_h)=v_h(\bx_j)-v_h(\bx_i)$.
Thus, the corresponding discrete problem is to find $u_h\in \mathcal V_1(\mathcal{K}_h)$ such that
\begin{equation}
\sum_{K\in \mathcal{K}_h}B_h^K(u_h,v_h)=F_h(v_h),\quad\forall v_h\in \mathcal V_1(\mathcal{K}_h).\label{discrete problem}
\end{equation}

\begin{remark}
We note that $\omega_{ij}^K=-a_h^K(\phi_i,\phi_j)$ in~\eqref{bilinear form Bh} is computable and already given in the virtual element methods for the Poisson equation.
If we choose $\omega_{ij}^K=-\mathbf{a}_h^K(\phi_i,\phi_j)$ as in~\eqref{eqn: fvm_discrete_prob} on a Voronoi mesh corresponding to a Delaunay triangulation consisting of acute triangles, the general bilinear form~\eqref{bilinear form Bh} becomes the one in the monotone EAVE method~\eqref{eqn: discrete problem_MEAVE}.
In practice, constant approximations $\alpha_{E_{ij}}$ and $\bm{\beta}_{E_{ij}}$ on $E_{ij}$ give a function $\psi_{E_{ij}}(\mathbf{x})={\alpha_{E_{ij}}}^{-1}\bm{\beta}_{E_{ij}}\cdot\mathbf{x}$ locally defined on $E_{ij}$.
Therefore, the bilinear form is computed as
\begin{align*}
\bar{\kappa}_{E_{ij}}\delta_{E_{ij}}^K(e^{\psi_{E_{ij}}} u_h)&=\alpha_{E_{ij}} \mathbb{B}({\alpha_{E_{ij}}}^{-1}\bm{\beta}_{E_{ij}}\cdot(\mathbf{x}_{i}-\mathbf{x}_{j}))u_h(\mathbf{x}_j)\\
&\qquad\qquad-\alpha_{E_{ij}} \mathbb{B}({\alpha_{E_{ij}}}^{-1}\bm{\beta}_{E_{ij}}\cdot(\mathbf{x}_{j}-\mathbf{x}_{i}))u_{h}(\mathbf{x}_i),
\end{align*}
where $\mathbb{B}(\cdot)$ is the Bernoulli function defined in~\eqref{bernoulli function}.
\end{remark}

\subsection{A sufficient condition for monotonicity property}
In addition, the following lemma presents a sufficient condition for the monotonicity property of the general framework of the EAVE methods.
\begin{lemma}\label{lemma: eave_monotone}
If the stiffness matrix for the Poisson equation with the virtual bilinear form~\eqref{Poisson bilinear form} is an $M$-matrix, then the stiffness matrix from the EAVE method~\eqref{discrete problem} is an $M$-matrix for $\alpha\in C^0(\bar{\Omega})$ and $\bm{\beta}\in \left(C^0(\bar{\Omega})\right)^2$.
\end{lemma}
\begin{proof}
It suffices to show that the off-diagonal components of the stiffness matrix are nonpositive.
For $i\not=j$, it is clear to see that $\mathbf{B}_{ij}=0$ if $\bx_i$ and $\bx_j$ are not the vertices of a polygon.
Otherwise, we let $\mathcal{K}_{ij}$ be the set of the polygons having both $\mathbf{x}_i$ and $\mathbf{x}_j$ as vertices.
Then, each component becomes
\begin{equation*}
\mathbf{B}_{ij}=\sum_{K\in\mathcal{K}_{ij}}B_h^{K}(\zeta_j,\zeta_i).
\end{equation*}
For any polygon $K\in\mathcal{K}_{ij}$, we have
\begin{equation*}
B_h^K(\zeta_j,\zeta_i) =-\omega_{ij}^{K}\bar{\kappa}_{E_{ij}}e^{\psi_{E_{ij}}(\mathbf{x}_j)}\quad\text{or}\quad -\omega_{ji}^{K}\bar{\kappa}_{E_{ji}}e^{\psi_{E_{ji}}(\mathbf{x}_j)}.
\end{equation*}
Note that $B_h^{K}(\zeta_j,\zeta_i)$ depends on the local indices of $i$ and $j$ which are determined counterclockwise in $K$.
Hence, since $\omega_{ij}^K= -a^K_h(\zeta_i,\zeta_j)\geq 0$ and $\bar{\kappa}_{E_{ij}}>0$, we get $\mathbf{B}_{ij}\leq 0$.
\end{proof}


\section{Numerical experiments}
\label{sec: numerical}
In this section, we present numerical experiments to demonstrate the effect of the monotonicity property and the optimal order of convergence on different mesh types.
First, in the virtual element space $\mathcal V_1(\mathcal{K}_h)$, we compute a mesh-dependent norm defined by 
the bilinear form~\eqref{Poisson bilinear form}, and we denote it as follows:
\begin{equation}
\left\|v_h\right\|_A:=\left(\sum_{K\in\mathcal{K}_h}a_h^K(v_h,v_h)\right)^{1/2},\quad\forall v_h\in \mathcal V_1(\mathcal{K}_h).\label{eqn: A-norm}
\end{equation}
This norm, which we call the $A$-norm, is equivalent to the $H^1$-norm from the property of the bilinear form in~\eqref{stability}. 
We also compute the mesh-dependent $L^\infty$-norm,
\begin{equation*}
\left\|v_h\right\|_{\infty }:=\max_{i}\left(\text{dof}_i(|v_h|)\right),\quad\forall v_h\in \mathcal V_1(\mathcal{K}_h),
\end{equation*}
to check the presence of spurious oscillations on numerical solutions.
The numerical experiments are implemented by authors' codes developed based on $i$FEM~\cite{CHE09}.
Furthermore, we consider the following example problem:
\begin{example}
Let the computational domain be $\Omega=(0,1)\times(0,1)$. The convection dominated problem in \cite{ELM05} is governed by,
\begin{subequations}\label{sys: example1}
\begin{alignat}{2}
-\nabla\cdot\left(\epsilon \nabla u+\bm{\beta}u\right)& =f && \quad \text{in } \Omega, \label{eqn: example1} \\
u &= g && \quad \text{on } \partial\Omega,  \label{eqn: bdexample1}
\end{alignat}
\end{subequations}
where $\epsilon\ll1$ is the diffusion coefficient, $\bm{\beta}=\langle 0,-1\rangle$, and $f\equiv0$.
The exact solution to the problem is given as
\begin{equation}
u(x,y)=x\left(\frac{1-e^{(y-1)/\epsilon}}{1-e^{-2/\epsilon}}\right),\label{eqn: example_exact}
\end{equation}
so it has sharper boundary layer around $y=1$ when $\epsilon$ gets smaller.
\label{ex1}
\end{example}


\subsection{Monotonicity test on Voronoi meshes with dual Delaunay}

In this test, we compare the edge-averaged virtual element methods~\eqref{eqn: discrete problem_MEAVE} and~\eqref{discrete problem} on Voronoi meshes with dual Delaunay triangulations consisting of acute triangles:
\begin{itemize}
    \item \texttt{M-EAVE}: \textit{Monotone} edge-averaged virtual element method~\eqref{eqn: discrete problem_MEAVE};
    \item \texttt{EAVE}: Edge-averaged virtual element method~\eqref{discrete problem}.
\end{itemize}
  \begin{figure}[h!]
\centering
\begin{tabular}{cccc}
\texttt{hexa-dual}  && \texttt{voro-dual}\\
    \includegraphics[width=.25\textwidth]{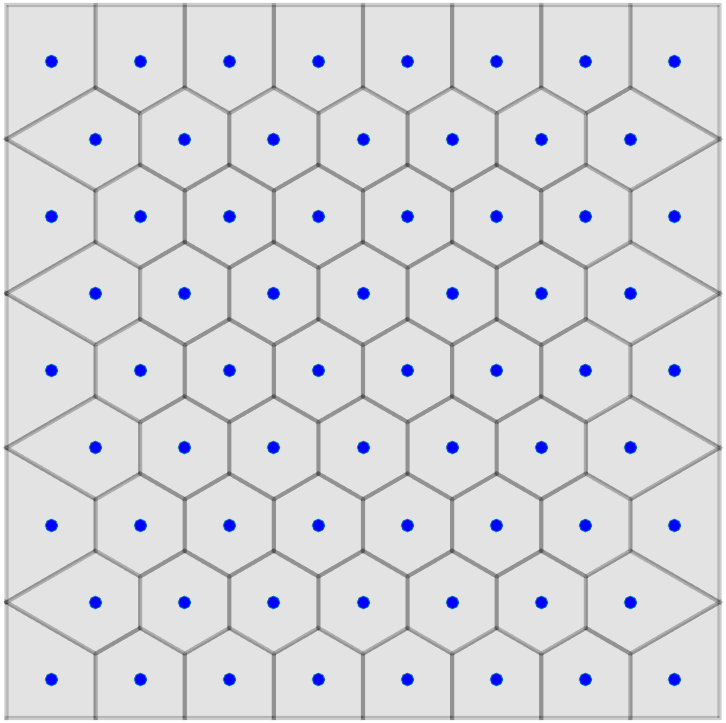}
    &
   &\includegraphics[width=.25\textwidth]{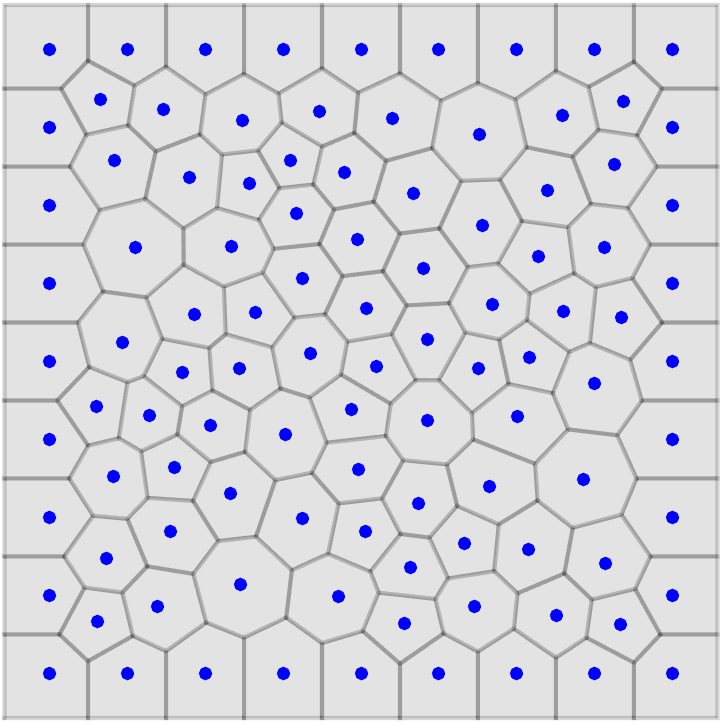}
   \end{tabular}
   \caption{Voronoi meshes corresponding to dual Delaunay triangulations with $h=2^{-3}$.}
   \label{figure: mesh type_dual}
\end{figure}
We use hexagonal and unstructured Voronoi meshes with dual Delaunay triangulations consisting of acute triangles.
In Figure~\ref{figure: mesh type_dual}, the blue dots mean vertices in a Delaunay triangulation, and it is easy to see that each triangle with the blue dots contains an internal Voronoi vertex.
\begin{table}[h!]
\footnotesize
    \caption{A mesh refinement study for \texttt{EAVE} and \texttt{M-EAVE} with $\epsilon=10^{-2}$ ($e_h=u_I-u_h$).}
    \label{table: mesh refinement_monotone}
    \centering
    \begin{tabular}{|c||c|c|c|c||c|c|c|c|}
    \hline
    &\multicolumn{8}{c|}{\texttt{hexa-dual}}\\
    \cline{2-9}   
    &\multicolumn{4}{c||}{\texttt{EAVE}} & \multicolumn{4}{c|}{\texttt{M-EAVE}}\\
    \cline{2-9}   
       $h$  & { $\norm{e_h}_A$} & {Order} &
       { $\norm{e_h}_\infty$} & {Order} &
       { $\norm{e_h}_A$} & {Order} &
       { $\norm{e_h}_\infty$} & {Order} \\ 
       \hline
       $2^{-3}$ & 1.631e-1 & - &  4.549e-2 & - & 1.728e-1 & - & 4.358e-2 & -\\
       \hline
       $2^{-4}$  & 1.058e-1 & 0.62 & 2.113e-2 & 1.11 & 9.922e-2 & 0.80 & 1.912e-2& 1.19  \\
       \hline
       $2^{-5}$   & 5.062e-2 & 1.06 & 7.153e-3 & 1.56 & 4.265e-2 & 1.22 & 5.871e-3 &1.70 \\
       \hline
       $2^{-6}$  & 1.737e-2 & 1.54 & 1.649e-3 & 2.12 & 1.430e-2 & 1.58 & 1.295e-3 & 2.18\\
       \hline
       \hline
    &\multicolumn{8}{c|}{\texttt{voro-dual}}\\
               \hline
    &\multicolumn{4}{c||}{\texttt{EAVE}} & \multicolumn{4}{c|}{\texttt{M-EAVE}}\\
                   \cline{2-9}   
        $h$ & {$\norm{e_h}_A$} & {Order} &
       {$\norm{e_h}_\infty$} & {Order} &
       {$\norm{e_h}_A$} & {Order} &
       {$\norm{e_h}_\infty$} & {Order} \\ 
       \hline
       $2^{-2}$ & 1.127e-1 & - &  3.291e-2 & - & 2.425e-1 & -  & 7.086e-2& -\\
       \hline
       $2^{-3}$  & 9.554e-2 & 0.24 & 1.829e-2 & 0.85 & 1.682e-1 & 0.53 & 2.922e-2 & 1.28 \\
       \hline
       $2^{-4}$   & 8.379e-2 & 0.19 & 9.731e-3 & 0.91 & 1.005e-1 & 0.74 & 8.397e-3 & 1.80\\
       \hline
       $2^{-5}$  & 4.399e-2 & 0.93 & 3.972e-3 & 1.29 & 4.647e-2 & 1.11 & 2.351e-3& 1.84\\
       \hline
    \end{tabular}
\end{table}
Hence, we solve the convection-dominated problem in Example~\ref{ex1} and perform a mesh refinement study for the \texttt{M-EAVE} and \texttt{EAVE} methods on the two Voronoi meshes with varying mesh sizes $h$ and $\epsilon =10^{-2}$.
Our test checks the $A$-norm error, $\norm{e_h}_A$ where $e_h=u_I-u_h$, equivalent to the $H^1$-error.
Table~\ref{table: mesh refinement_monotone} shows that the \texttt{M-EAVE} and \texttt{EAVE} methods have first-order convergence in the $A$-norm, and the $L^\infty$-errors tend to converge to zero of almost second-order.
We see from Table~\ref{table: mesh refinement_monotone} that the \texttt{M-EAVE} method performs better in the orders of convergence and magnitude of the errors in most cases.
Thus, we conclude that the \texttt{M-EAVE} method satisfying the monotonicity property is beneficial for solving convection-dominated problems on such Voronoi meshes.


\subsection{Accuracy test on general polygonal meshes}
This test confirms that the \texttt{EAVE} method~\eqref{discrete problem} works with general polygonal meshes by checking the accuracy.
In order to confirm the accuracy of the \texttt{EAVE} method, we conduct a mesh refinement study with different mesh sizes $h$ and the fixed $\epsilon=10^{-2}$ on various polygonal meshes.
We also solve the problem in Example~\ref{ex1} on the following mesh types (see also Figure~\ref{figure: mesh type} as an example):
\begin{itemize}
\item \texttt{voro}: a Voronoi tessellation;
\item \texttt{opti}: a Voronoi tessellation where the cell shapes are optimized via a Lloyd algorithm~\cite{TAL12};
\item \texttt{ncvx}: a mesh composed by structured non-convex polygons and triangles.
\end{itemize}
\begin{figure}[h!]
\centering
\begin{tabular}{cccccc}
\texttt{voro}  & \texttt{opti}
   &\texttt{ncvx}\\
    \includegraphics[width=.25\textwidth]{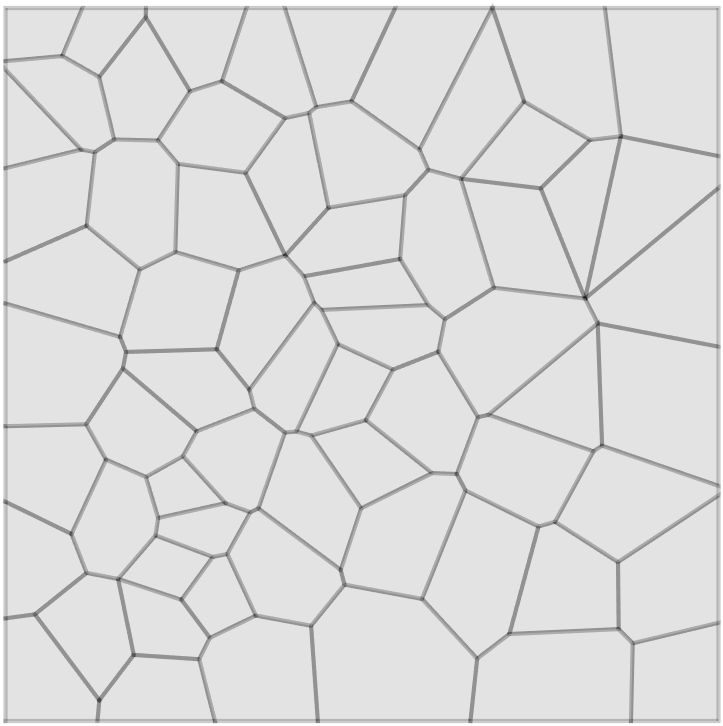}
   &\includegraphics[width=.25\textwidth]{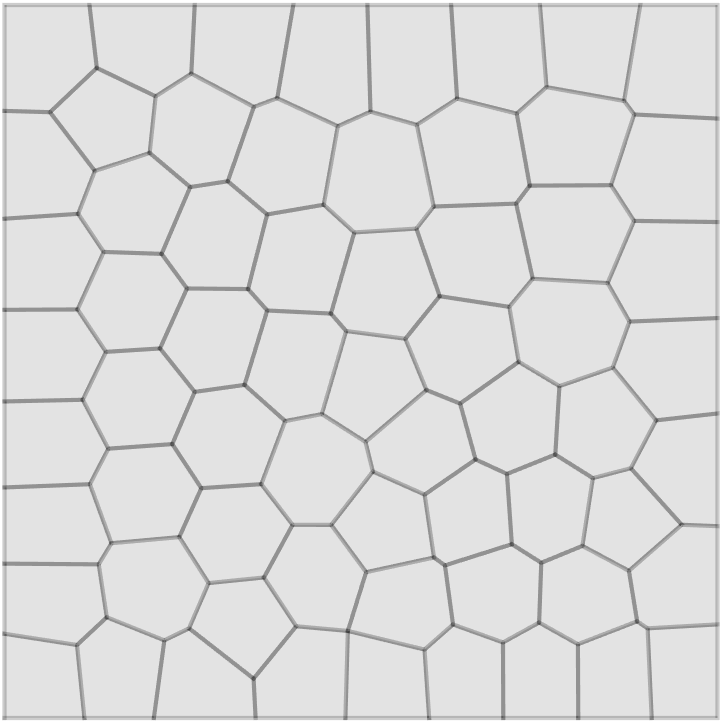}
   &\includegraphics[width=.25\textwidth]{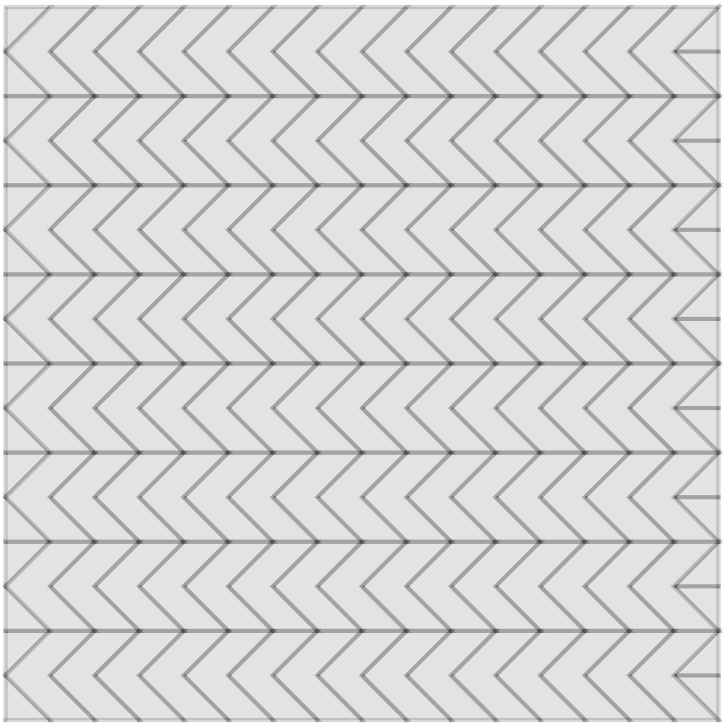}
   \end{tabular}
   \caption{Various polygonal meshes with $h=2^{-3}$.}
   \label{figure: mesh type}
\end{figure}

Table~\ref{table: mesh refinement} shows the \texttt{EAVE} method's $A$-norm error and $L^\infty$-error for the \texttt{voro}, \texttt{opti}, and \texttt{ncvx} meshes.
\begin{table}[!h]
\footnotesize
\caption{A mesh refinement study for \texttt{EAVE} on various meshes with $\epsilon=10^{-2}$.}
    \label{table: mesh refinement}
    \centering
    \begin{tabular}{|c||c|c||c|c||c|c|}
    \hline
        &  \multicolumn{2}{c||}{\texttt{voro}
        } & 
         \multicolumn{2}{c||}{\texttt{opti}
        } &
        \multicolumn{2}{c|}{\texttt{ncvx}}\\
    \cline{2-7}   
       $h$  & {$\norm{u_I-u_h}_A$} & {Order} &
       {$\norm{u_I-u_h}_A$} & {Order} &
       {$\norm{u_I-u_h}_A$} & {Order} \\ 
       \hline
       $2^{-4}$ & 2.383e-1 & - &  1.078e-1 & - & 1.355e-1 & - \\
       \hline
       $2^{-5}$  & 1.586e-1 & 0.59 & 6.250e-2 & 0.79 & 2.488e-2 & 2.45  \\
       \hline
       $2^{-6}$   & 8.057e-2 & 0.98 & 2.914e-2 & 1.10 & 4.987e-3 & 2.32  \\
       \hline
       $2^{-7}$  & 4.263e-2 & 0.92 & 1.414e-2 & 1.04 & 9.253e-4 & 2.43 \\
       \hline
       $2^{-8}$  & 2.126e-2 & 1.00 & 6.902e-3 & 1.03 & 1.665e-4 & 2.47 \\
       \hline
       \hline
        $h$ & {$\norm{u_I-u_h}_\infty$} & {Order} &
       {$\norm{u_I-u_h}_\infty$} & {Order} &
       {$\norm{u_I-u_h}_\infty$} & {Order} \\ 
       \hline
       $2^{-4}$ & 3.714e-2 & - &  1.139e-2 & - & 2.563e-2 & - \\
       \hline
       $2^{-5}$  & 1.494e-2 & 1.31 & 5.218e-3 & 1.13 & 3.551e-3 & 2.85 \\
       \hline
       $2^{-6}$   & 4.183e-3 & 1.84 & 1.797e-3 & 1.54 & 5.052e-4 & 2.81 \\
       \hline
       $2^{-7}$  & 1.224e-3 & 1.77 & 4.414e-4 & 2.03 & 6.625e-5 & 2.93 \\
       \hline
       $2^{-8}$  & 3.123e-4 & 1.97 & 1.140e-4 & 1.95 & 8.424e-6 & 2.98 \\
       \hline
    \end{tabular}
\end{table}
On all the given meshes, the orders of convergence for the $A$-norm errors are at least first-order.
Also, the $L^\infty$-errors converge to zero of almost second-order.
For the \texttt{ncvx} mesh, both $A$-norm error and $L^\infty$-error show super-convergence because the \texttt{ncvx} mesh consists of structured grid points and the points are aligned with the boundary layer.
Therefore, these numerical results confirm that the \texttt{EAVE} method produces accurate numerical solutions on general polygonal meshes.

\subsection{Test for convection-dominated case}

In this test, we check the performance of the \texttt{M-EAVE} and \texttt{EAVE} methods compared to other stabilized numerical methods for convection-dominated problems.
Various numerical methods having the same order of convergence are employed to solve the problem in Example~\ref{ex1}, and they are denoted as follows:
\begin{itemize}
    \item \texttt{FE}: Conforming finite element method with piecewise linear functions;
    \item \texttt{SUPG}: Streamline-upwind Petrov-Galerkin method~\cite{BRO82} with piecewise linear functions and the stability parameter
    $$s_E=\frac{0.25h^2_T}{\epsilon P_e}\left(1-\frac{1}{P_e}\right)\quad\text{and}\quad P_e=\frac{\beta_Th_T}{2\epsilon},$$
    where $\beta_T=\|\bm{\beta}\|_{L^\infty(T)}$ and $h_T=|T|^{1/2}$;
    \item \texttt{EAFE}: Edge-averaged finite element scheme~\eqref{eqn: EAFE_discrete problem};
    \item \texttt{M-EAVE}: Monotone edge-averaged virtual element method~\eqref{eqn: discrete problem_MEAVE};
    \item \texttt{EAVE}: Edge-averaged virtual element method~\eqref{discrete problem}.
\end{itemize}

The performance of the stabilized numerical methods can be explained by the presence of unexpected oscillations in their numerical solutions.
According to the exact solution~\eqref{eqn: example_exact},  numerical solutions are expected to capture the boundary layer without oscillation.
However, even though numerical methods imply small $H^1$- or $L^2$-errors, nonphysical oscillations may locally occur on their numerical solutions because they may not satisfy the discrete maximum principle with the given mesh size.
The $L^\infty$-error, $\norm{u_I-u_h}_\infty$, can identify such nonphysical oscillations, where $u_I$ is the nodal value interpolant of $u$.
For comparisons across the stabilized numerical methods, we solved the problem in Example~\ref{ex1} with varying $\epsilon$ values and mesh size $h$, respectively.
\begin{figure}[h!]
    \centering
\includegraphics[width=.35\textwidth]{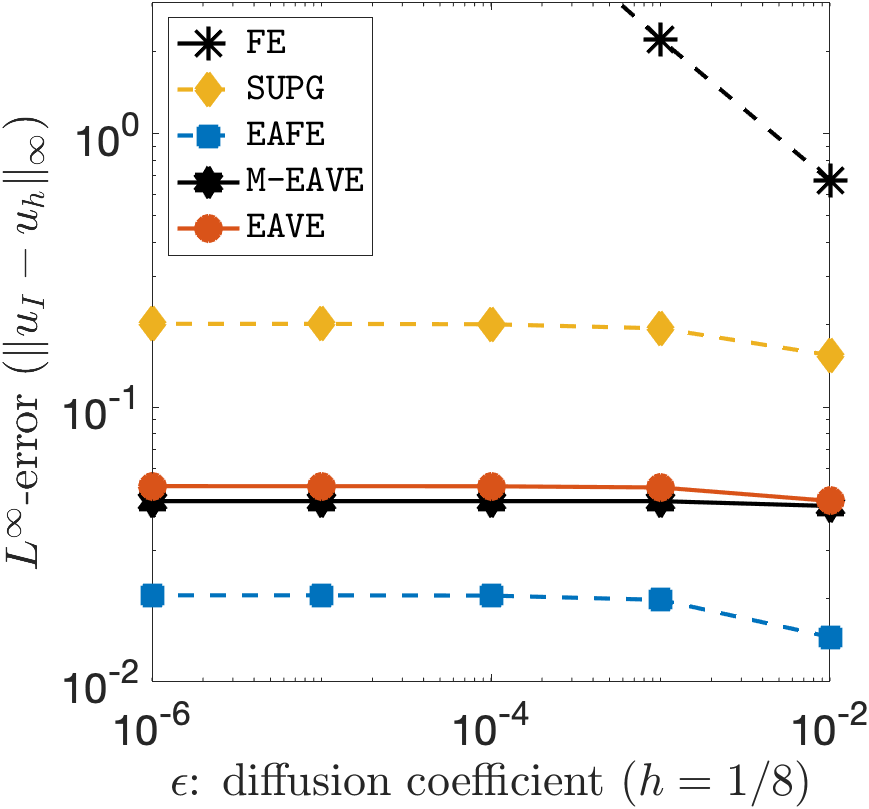}
    \hskip 20pt
\includegraphics[width=.35\textwidth]{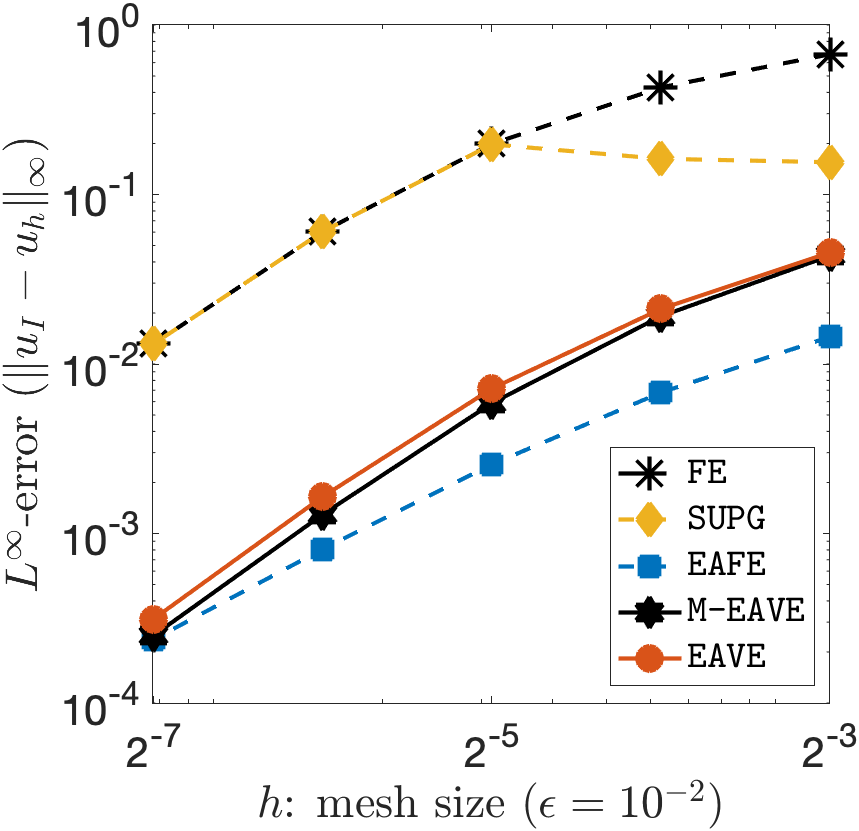}
    \caption{Error profiles of the numerical methods with varying $\epsilon$ (left) and $h$ (right).}
    \label{figure: comparison e&h}
\end{figure}
Figure~\ref{figure: comparison e&h} shows the $L^\infty$-errors of the numerical methods in two ways; (i) with varying $\epsilon$ values from $10^{-2}$ to $10^{-6}$ and the fixed mesh size $h=2^{-3}$, (ii) with different mesh size from $2^{-3}$ to $2^{-7}$ and the fixed $\epsilon=10^{-2}$.
As $\epsilon$ decreases, the $L^\infty$-errors of the \texttt{FE} method rapidly increase, which shows the need for stabilization techniques for convection-dominated cases.
On the other hand, the $L^\infty$-errors of the stabilized methods (\texttt{SUPG}, \texttt{EAFE}, \texttt{M-EAVE}, and \texttt{EAVE}) seem independent of $\epsilon$.
As $h$ decreases, all the methods produce the $L^\infty$-errors decreasing in a similar order, but the errors of edge-averaged methods (\texttt{EAFE}, \texttt{M-EAVE}, and \texttt{EAVE}) are almost a hundred times smaller than the \texttt{SUPG} method.

In order to perform a quantitative comparison, we display the numerical solutions obtained by the \texttt{FE}, \texttt{SUPG}, \texttt{EAFE}, and \texttt{M-EAVE} methods with $\epsilon=10^{-2}$ and $h=2^{-4}$ in Figure~\ref{figure: numesol_ex1}. Since the solution of the \texttt{EAVE} method is almost the same as the \texttt{M-EAVE}, we omit \texttt{EAVE} here.
\begin{figure}[h!]
\centering
\begin{tabular}{cccc}
\texttt{ FE}  & \texttt{ SUPG}
   &\texttt{ EAFE}&\hskip 10pt\texttt{ M-EAVE}\\
    \includegraphics[width=.21\textwidth]{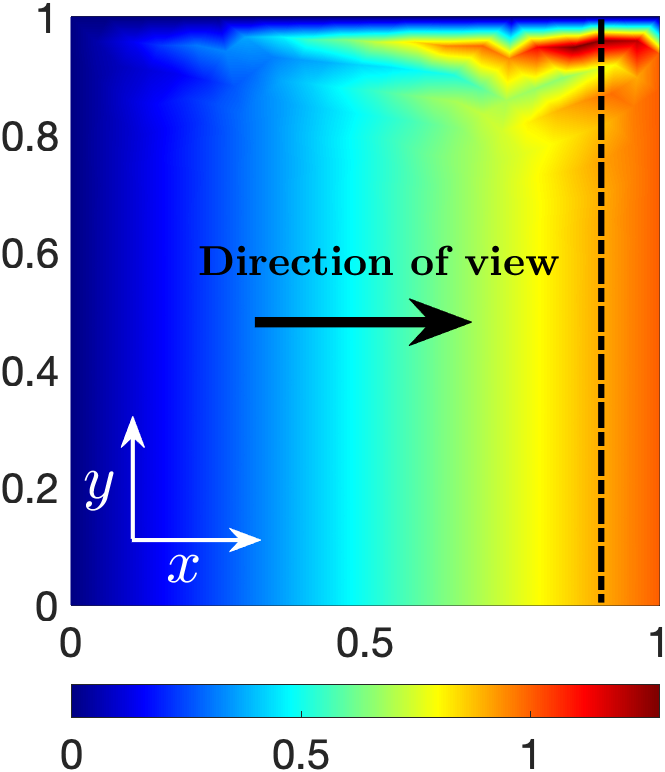}
   &\includegraphics[width=.21\textwidth]{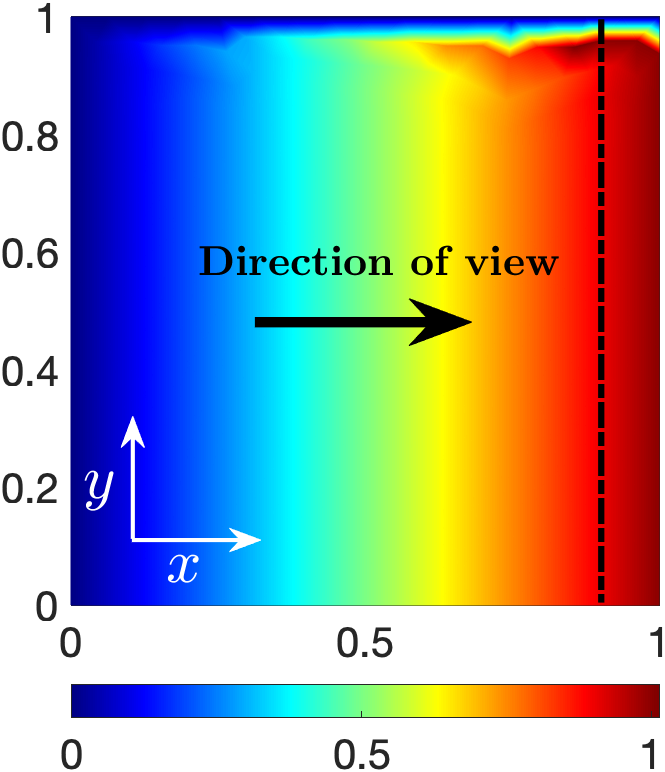}
   &\includegraphics[width=.21\textwidth]{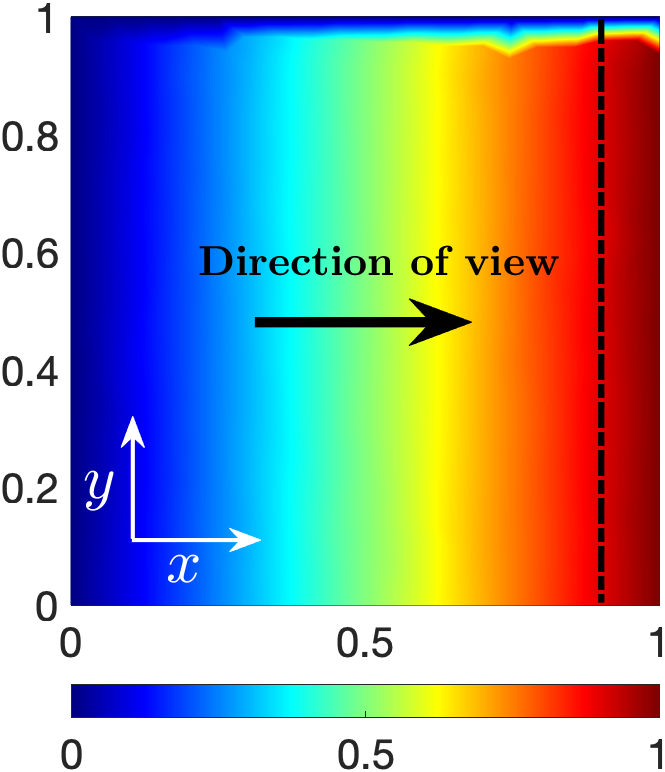}
   &\includegraphics[width=.21\textwidth]{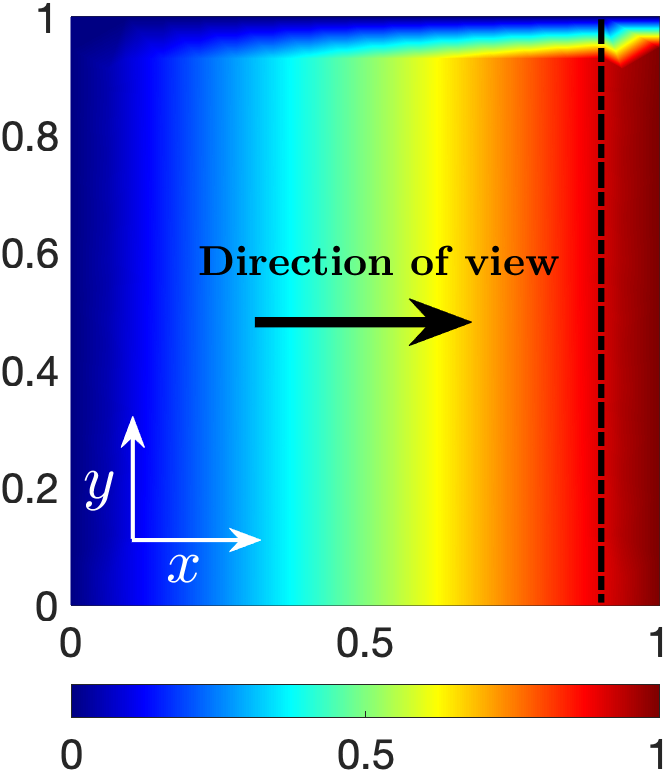}
   \end{tabular}\\
   {\footnotesize (a) Numerical solutions for the methods}\\
   \vskip 5pt
   \begin{tabular}{cccc}
   \texttt{ FE}  &\texttt{ SUPG}
   &\texttt{ EAFE} &\hskip 10pt\texttt{ M-EAVE}\\
   \includegraphics[width=.21\textwidth]{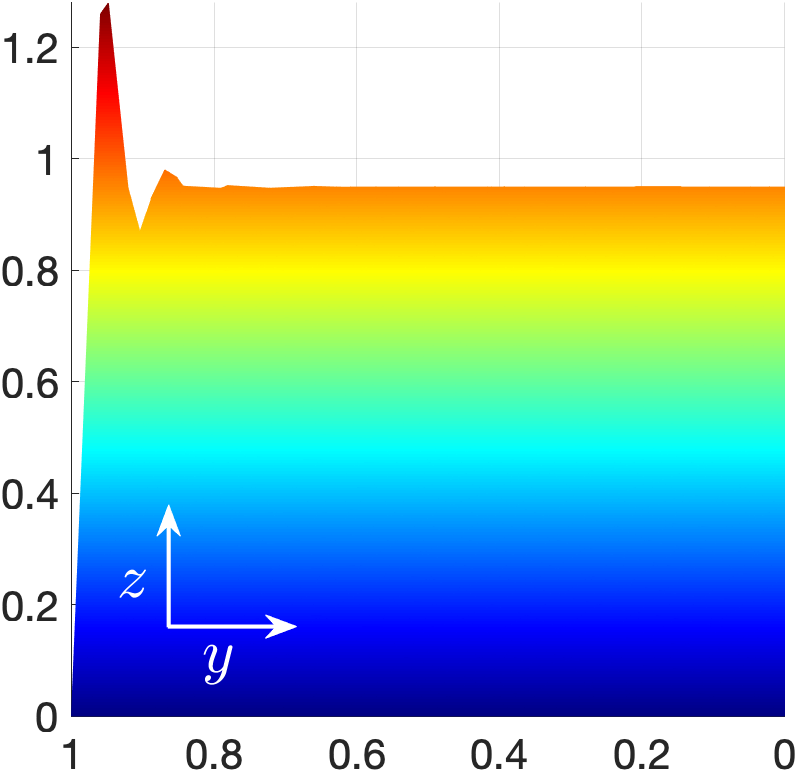}
   &\includegraphics[width=.21\textwidth]{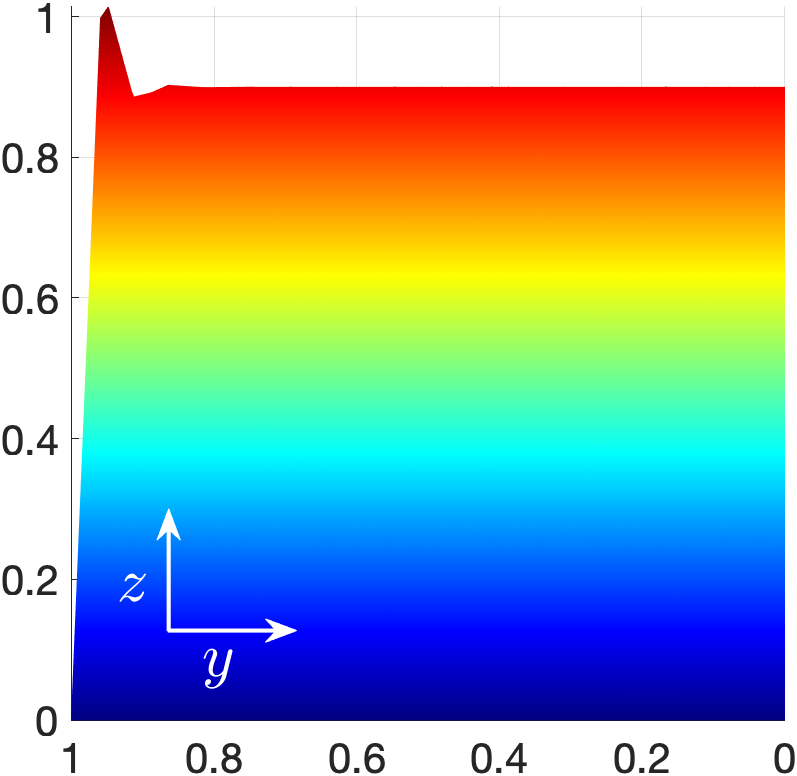}
   &\includegraphics[width=.21\textwidth]{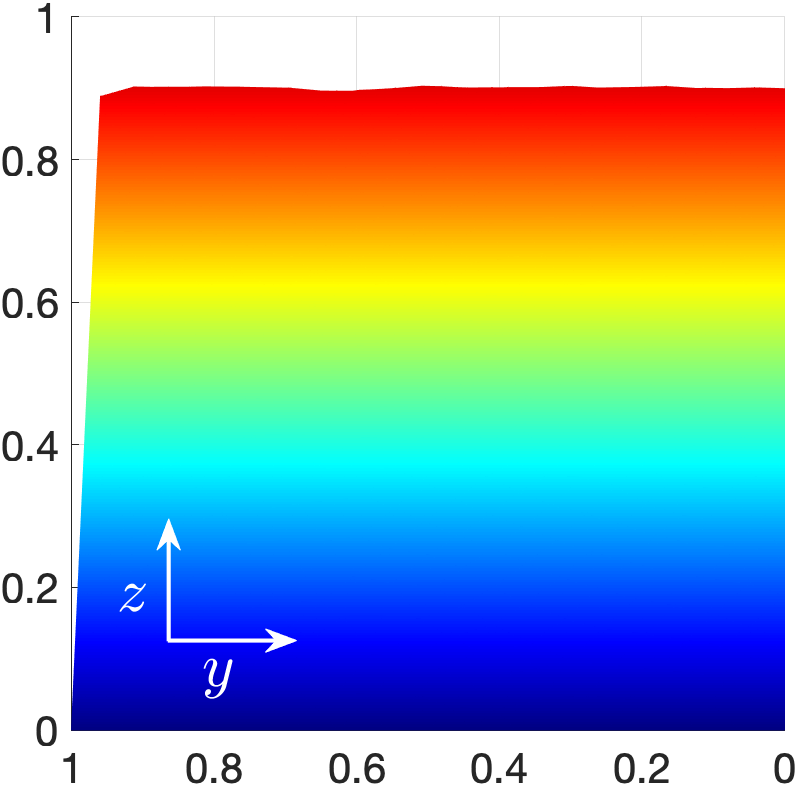}
   &\includegraphics[width=.21\textwidth]{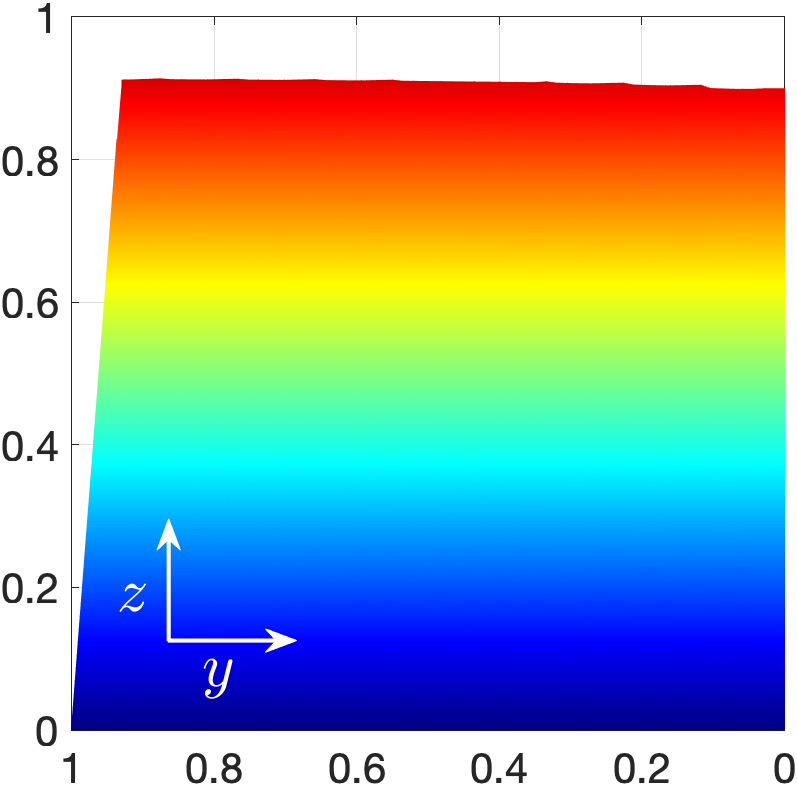}
\end{tabular}\\
{\footnotesize (b) Cross sections of the numerical solutions from $x=0$ to $x=0.9$}
    \caption{Comparison of the numerical solutions with $\epsilon=10^{-2}$ and $h=2^{-4}$.}
    \label{figure: numesol_ex1}
\end{figure}
As expected, the numerical solution for the \texttt{FE} method contains spurious oscillations around the boundary layer, and the magnitude of the biggest oscillation is over 20\% of the maximum of the exact solution.
The \texttt{SUPG}, a stabilized method, provides a more stable numerical solution than the \texttt{FE} method. However, the solution still contains a nonphysical oscillation whose magnitude is over 10\% of the exact solution.
These spurious oscillations may deteriorate the quality of the numerical solutions and cause difficulty in capturing hidden boundary layers.
On the other hand, the \texttt{EAFE} and \texttt{M-EAVE} methods satisfy the monotonicity property, so they produce numerical solutions without such oscillations and nicely capture the sharp boundary layer.
Therefore, we confirm that the \texttt{M-EAVE} and \texttt{EAVE} methods produce a quality approximation on polygonal meshes.

\section{Conclusion}
\label{sec: conclusion}
We have developed edge-averaged virtual element methods for convection-diffusion problems by generalizing the edge-averaged finite element  stabilization~\cite{XU99} to the virtual element methods. A proper flux approximation and mass lumping allow us to obtain a monotone EAVE method on a Voronoi mesh with a dual Delaunay triangulation consisting of acute triangles. The Bernoulli function and the geometric data of such a Voronoi mesh readily compute the bilinear form in the monotone method, while the monotone method guarantees stable numerical solutions. We also presented a general framework for EAVE methods, proposing another EAVE bilinear form working with general polygonal meshes. Our numerical experiments demonstrated the effect of the monotonicity property and the optimal order of convergence on different mesh types.

Extending the EAVE methods to three-dimensional cases will be one of our future research directions. The EAVE methods can be also used for application problems that include convection dominance and require polygonal meshes.

	\bibliography{EAVE}
	
\end{document}